\documentclass{amsart}
\usepackage[utf8]{inputenc}
\usepackage[english]{babel}
\usepackage[makeroom]{cancel}
\usepackage{amsmath}
\usepackage{amsfonts}
\usepackage{amssymb}
\usepackage{mathtools}
\usepackage{pgf,tikz,pgfplots} 
\pgfplotsset{compat=1.15}
\usepackage{graphicx}
\usepackage{float}
\usepackage{mathrsfs}
\usetikzlibrary{arrows.meta,positioning,automata,shadows}
\usepackage{colortbl}
\usepackage{subfigure}
\usepackage{hyperref}
\usepackage{comment}

\usepackage[most]{tcolorbox}

\newcommand{\negfrac}[1]{[\![#1]\!]}

\newtheorem{thm}{Theorem}[section]
\newtheorem*{thm*}{Theorem}
\newtheorem{lemme}[thm]{Lemma}
\newtheorem{coro}[thm]{Corollary}
\newtheorem{prop}[thm]{Proposition}

\theoremstyle{definition}
\newtheorem{defi}[thm]{Definition}

\theoremstyle{remark}
\newtheorem{nota}[thm]{Notation}

\newtheorem{rem}[thm]{Remark}
\newtheorem{ex}[thm]{Example}

\newcommand{\abs}[1]{\left\lvert#1\right\rvert}

\DeclareMathOperator{\Id}{Id}

\DeclareMathOperator{\id}{id}

\DeclareMathOperator{\Stab}{Stab}
\DeclareMathOperator{\ev}{ev}
\DeclareMathOperator{\GL}{GL}
\DeclareMathOperator{\PGL}{PGL}
\DeclareMathOperator{\Gal}{Gal}
\DeclareMathOperator{\Tr}{Tr}

\DeclareMathOperator{\PSL}{PSL}

\DeclareMathOperator{\SL}{SL}

\newcommand{\RR}{\mathbb{R}}

\newcommand{\NN}{\mathbb{N}}

\newcommand{\ZZ}{\mathbb{Z}}
\newcommand{\QQ}{\mathbb{Q}}
\newcommand{\PP}{\mathbb{P}}

\begin{document}

\title{Symmetries of the $q$-deformed real projective line}
\author{Perrine Jouteur}
\date{}

\address{
Perrine Jouteur,
Laboratoire de Math\'ematiques de Reims, UMR~9008 CNRS et
Universit\'e de Reims Champagne-Ardenne,
U.F.R. Sciences Exactes et Naturelles,
Moulin de la Housse - BP 1039,
51687 Reims cedex 2,
France} 
\email{perrine.jouteur@univ-reims.fr}

\begin{abstract}
We generalize in two steps the quantized action of the modular group on $q$-deformed real numbers introduced by Morier-Genoud and Ovsienko in \cite{MGO}. First, we let the projective general linear group $\PGL_2(\ZZ)$ act on $q$-real numbers via a $q$-deformed action. The deformed matrices we get have combinatorial interpretations, and we show that their traces are palindromic polynomials. Then we consider an extension of the group $\PGL_2(\ZZ)$ by the $2$-elements cyclic group, and define a deformed action of this extension on $q$-real numbers. We deduce from these actions some underlying relations between $q$-real numbers, and between left and right versions of $q$-deformed rational numbers. In particular we investigate the case of some algebraic numbers of degree $4$ and $6$. We also prove that the deformation of real numbers is an injective process.
\end{abstract}

\maketitle
\vspace{-0.5 cm}
\tableofcontents

\newpage

\section{Introduction and main results}

The $q$-deformed real numbers were defined by Morier-Genoud and Ovsienko in \cite{MGO} and \cite{MGOr}, using a deformation of the action of the modular group $\PSL_2(\ZZ)$ on the rational projective line. These $q$-numbers happen to be related to many different areas, including frieze patterns \cite{MGO_friezes}, Farey triangulations \cite{leclere_quantum_2023,ren_2024}, combinatorics of fence posets \cite{McConville,oguz_rank_2023}, snake graphs \cite{Ovenhouse,Higher_q_fractions}, cluster algebras \cite{MGO}, Markov numbers and Diophantine approximation \cite{Kogiso_Markov,labbe_2022,Oguz_Markov}, knot theory \cite{sikora_tangle_2024,Kogiso_Wakui,wakui_coprime_2022} and homological algebra \cite{BBL,fan_topological_2023}.\\

In this paper we extend the group of symmetries acting on the $q$-real numbers.
Given a real number $x$ and its $q$-deformation $[x]_q$, by Morier-Genoud and Ovsienko's construction one has the following relations:
\begin{equation}
\label{formulas}
\begin{array}{c}
[x+1]_q = q[x]_q + 1 ,\\[4pt]
 \text{ and } \\[4pt]
\left[\frac{-1}{x}\right]_q = -\frac{1}{q[x]_q}.
\end{array}
\end{equation}

\noindent We show that in addition the following holds
\begin{equation}
\forall x\in \RR\setminus \QQ, ~ [-x]_q = \frac{-[x]_q+1-q^{-1}}{(q-1)[x]_q + 1} ~~~\text{ and equivalently, }~~~  \left[\frac{1}{x}\right]_q = \frac{(q-1)[x]_q + 1}{q[x]_q + 1 -q}.
\end{equation}
Similar formulas hold for rational numbers, but this case is more intricate and will be adressed in Section~\ref{sec:qPGL}. We get a deformation of the action of $\PGL_2(\ZZ)$ on the real projective line.

As an application of these extended symmetries of $q$-numbers, we get convenient formulas to work with $q$-deformed positive continued fractions (whereas in previous works the negative continued fractions were preferred, see \cite{Leclere_modular}). In particular, we show the injectivity of the deformation map.

We also use the deformed action of $\PGL_2(\ZZ)$ to deform some polynomial equations whose Galois groups act on roots by fractional-linear transformations, following the ideas of \cite{Ovsienko_Ustinov}.

Besides, we define a $q$-action of the central extension $\PGL_2(\ZZ)\times \ZZ_2$,  where $\ZZ_2 := \{1,-1\}$ is the group with two elements, on the rational projective line. This action exchanges the left and the right versions of $q$-rational numbers. For every rational number $x \in \QQ$, we deduce the following relationship between left and right $q$-rational numbers, which was first observed by Thomas in \cite{thomas_2024}.

\begin{equation}
[x]^{\sharp}_q = \frac{[x]^{\flat}_{q^{-1}} + q-1 }{(1-q)[x]^{\flat}_{q^{-1}} + q} \text{ and } [x]^{\flat}_q = \frac{[x]^{\sharp}_{q^{-1}} + q-1 }{(1-q)[x]^{\sharp}_{q^{-1}} + q}.
\end{equation}
\noindent
The extended $q$-action provides a coherent framework explaining the relations
\begin{equation}
\forall x\in \QQ, ~[-x]_q = -q^{-1}[x]_{q^{-1}} \text{ and } \left[\frac{1}{x}\right]_q = \frac{1}{[x]_{q^{-1}}}.
\end{equation}

\subsection{Background on $q$-deformed numbers}

Let $q$ be a formal parameter. The classical $q$-deformed integers are polynomials in $q$,
$$
[n]_q = \frac{1-q^n}{1-q} = 1+q+\cdots q^{n-1} \in \ZZ[q].
$$
More generally, $q$-deformed rational numbers are rational fractions in $q$, belonging to $\ZZ(q)$, and $q$-deformed irrational numbers are formal Laurent series in $q$, belonging to $\ZZ[[q]][q^{-1}]$, where

$$
\ZZ[[q]][q^{-1}] := \left\lbrace \sum_{i \geq \nu} a_i q^i ~|~ a_i \in \ZZ, ~\nu \in \ZZ \right\rbrace.
$$

First, let us focus on $q$-deformed rational numbers. We will consider the extended rational line, $\PP^1(\QQ) = \QQ\cup \{\infty\}$. The modular group $\PSL_2(\ZZ)$ acts on $\PP^1(\QQ)$ by fractional-linear transformations,
$$
\begin{array}{c c c}
\PSL_2(\ZZ)\times \PP^1(\QQ) & \longrightarrow & \PP^1(\QQ) \\[4pt]
\begin{pmatrix}
a & b \\
c & d \\
\end{pmatrix} \cdot \frac{r}{s} & = & \frac{ar + bs}{cr+ds}.\\
\end{array}
$$

The $q$-deformed recursive relation $[x+1]_q = q[x]_q + 1$ provides a way of deforming the translation $T : x \mapsto x+1$, as $T_q : [x]_q \mapsto q[x]_q+1$. Then the whole action of $\PSL_2(\ZZ)$ can be deformed using the two generators

\begin{equation}
\label{generators}
T = \begin{pmatrix}
		1 & 1\\
		0 & 1\\
		\end{pmatrix}  \text{ and } S = \begin{pmatrix}
		0 & -1\\
		1 & 0\\
		\end{pmatrix}
\end{equation}
\noindent becoming
\begin{equation}
\label{qgenerators}
T_q = \begin{pmatrix}
		q & 1\\
		0 & 1\\
		\end{pmatrix}  \text{ and } S_q = \begin{pmatrix}
		0 & -q^{-1}\\
		1 & 0\\
		\end{pmatrix}.
\end{equation}

In $\PSL_2(\ZZ)$, the defining relations between $T$ and $S$ are $S^2 = (TS)^3 = \Id$. After the deformation, the relations are unchanged : $S_q^2 = -q^{-1}\Id$ and $(T_qS_q)^3 = -\Id$. The subgroup of $\PGL_2(\ZZ[q,q^{-1}])$ generated by $T_q$ and $S_q$ is thus isomorphic to the modular group \cite{Leclere_modular}, so this gives a well defined $q$-deformation of the action of $\PSL_2(\ZZ)$  on $\PP^1(\QQ)$. We will denote by $\PSL_{2,q}(\ZZ)$ the image of $\PSL_2(\ZZ)$ in $\PGL_2(\ZZ[q,q^{-1}])$ by the group homomorphism sending $T$, $S$ on $T_q$, $S_q$.

$$
\begin{array}{c c c}
\PSL_2(\ZZ) & \overset{\simeq}{\longrightarrow} & \PSL_{2,q}(\ZZ) \subset \PGL_2(\ZZ[q,q^{-1}]) \\
M & \longmapsto & M_q \\
\end{array}.
$$

Note that the group $\PSL_{2,q}(\ZZ) $ can also be understood via the reduced Burau representation of the braid group on three strands $B_3$. Let $\sigma_1$ and $\sigma_2$ be the two classical generators of $B_3$. The Burau representation of $B_3$ is 
$$
\rho_3 : B_3 \longrightarrow \SL_2(\ZZ(q)),
$$
$$
\sigma_1 \longmapsto \begin{pmatrix}
									q & 1\\
									0 & 1\\
									\end{pmatrix} \text{ and }
\sigma_2 \longmapsto \begin{pmatrix}
									1 & 0\\
									-q & q\\
									\end{pmatrix}.
$$ 
\noindent The $q$-modular group is then precisely the projective version of the image of $B_3$ via this representation. Indeed, the centre $Z(B_3)$ of $B_3$ is the kernel of $\rho_3$, and the induced quotient homomorphism is an isomorphism of groups such that
$$
\begin{array}{c c c}
B_3/Z(B_3) & \overset{\simeq}{\longrightarrow} & \PSL_{2,q}(\ZZ)\\
\sigma_1 & \longmapsto & T_q\\
\sigma_1\sigma_2\sigma_1 & \longmapsto & S_q\\
\end{array}.
$$

\noindent This point of view was used in \cite{morier-genoud_burau_2024} to give a partial answer to the question of faithfulness of specialization of the Burau representation. It was also used in \cite{Jouteur_Burau} to deform the rational projective plane.\\

As the action by fractional-linear transformations on $\PP^1(\QQ)$ is transitive, the $q$-deformation of one point, say $\infty = \frac{1}{0}$, is enough to have the $q$-deformation of every point of the projective rational line. Yet, there is a subtlety: for an orbit of $\PP^1(\ZZ(q))$ to be an actual $q$-deformation of $\PP^1(\QQ)$, there must be in the orbit exactly one $q$-deformation for each rational number.

\begin{prop}
	\label{two_deformations}
There are only two $q$-deformations of $\infty$ in $\PP^1(\ZZ(q))$ whose orbits under the action of $\PSL_{2,q}(\ZZ)$ satisfy the condition:
\begin{center}
``For every $x \in \PP^1(\QQ)$, there is exactly one element $f_x(q)$ of the orbit such that $f_x(1) = x$".
\end{center}
These two $q$-deformations of $\infty$ are 
$$
[\infty]^{\sharp}_q = \frac{1}{0} \text{ and } [\infty]^{\flat}_q = \frac{1}{1-q}.
$$
\end{prop}

This property, which will be proved in Appendix~\ref{appendixb}, leads to the definition of left and right versions of $q$-rational numbers. 

\begin{defi}
Let $x\in \PP^1(\QQ)$ and let $M \in \PSL_2(\ZZ)$ be such that $M\cdot \infty = x$. The two deformations of $x$ are elements of $\PP^1(\ZZ(q))$, given by 
\begin{itemize}
\item[$\bullet$] Right version: $[x]^{\sharp}_q = M_q \cdot \frac{1}{0}$,
\item[$\bullet$] Left version: $[x]^{\flat}_q = M_q \cdot \frac{1}{1-q}$.
\end{itemize} 
\end{defi}

In the rest of the article, except when necessary, we will skip the subscript $``q"$ in the notation of $q$-rational numbers, writing $[x]^{\sharp}$ instead of $[x]_q^{\sharp}$ and $[x]^{\flat}$ instead of $[x]_q^{\flat}$.

Note that the right version is the original one, defined by Morier-Genoud and Ovsienko \cite{MGO}. In particular, it recovers the usual $q$-integers of Gauss: for $n\in \NN$,
$$
[n]^{\sharp} = \frac{q^n-1}{q-1} = 1+q+q^2 +\cdots +q^{n-1}.
$$

\noindent The left version was noticed in \cite[Remark 3.2]{MGO} and defined in \cite{BBL} by Bapat, Becker and Licata. The names left and right refer to the different behaviour of a sequence of $q$-rational numbers when converging from the left or from the right to a rational number. The left $q$-integers are:
$$
[n]^{\flat} = \frac{q^{n+1}-q^n + q^{n-1} - 1}{q-1} = 1+q+q^2 + \cdots +q^{n-2} + q^n.
$$

The $q$-deformation of a real irrational number $x$ relies on the approximation of $x$ by any sequence of rational numbers converging to $x$. The Taylor series coming from the $q$-deformations of the rational numbers approaching $x$ stabilize and converge toward a formal Laurent series denoted by $[x]_q$ and being the $q$-deformation of $x$ by definition \cite{MGOr}. For more precise statements, see Section~\ref{sec:qreals}.

We get two deformations maps that coincide on irrational numbers:

$$
\begin{array}{c c c c}
Q^{\square} : & \RR & \longrightarrow & \ZZ[[q]][q^{-1}] \\
			& x & \longmapsto & \begin{cases}
								[x]^{\square} \text{ if } x\in \QQ\\
								[x]_q \text{ else}
								\end{cases}
\end{array}, \text{   with } \square \in \{\sharp, \flat\}. 
$$

\begin{prop}
\label{injective}
The two deformations maps $Q^{\sharp}$ and $Q^{\flat}$ are injective.
\end{prop}

This proposition will be proved in Section~\ref{subsec:injectivity}.

\subsection{First extension of symmetries}

The $q$-action of $\PGL_2(\ZZ)$ on $\PP^1(\RR)$ relies on the $q$-deformation of the transformation $x\mapsto -x$, corresponding to the matrix $N = \begin{pmatrix}
-1 & 0\\
0 & 1\\
\end{pmatrix}$, by 
\begin{equation}
N_q  = \begin{pmatrix}
				-1 & 1-q^{-1}\\
				q-1 & 1\\
				\end{pmatrix}.
\end{equation}

We prove that this is the only way of deforming $N$. Note that $N_q$ has determinant $-q^{-1}(q^2-q+1)$. Let us then set 
$$
t= q^2 -q +1.
$$
We will work up to powers of $t$, so we denote by $\Lambda$ the localization of $\ZZ[q,q^{-1}]$ by $t$, and $\overline{\Lambda}$ its field of fractions.
$$
\Lambda := \ZZ[q,q^{-1}]_t \text{ , } \overline{\Lambda} = \ZZ(q).
$$

The relations between $T,S,N$ in $\PGL_2(\ZZ)$ are $S^2 = (TS)^3 = 1$ and $N^2 = (NT)^2 = (NS)^2 = 1$. We prove that they remain unchanged if we replace $T,S,N$ by $T_q,S_q,N_q$ in $\PGL_2(\Lambda)$. The subgroup of $\PGL_2(\Lambda)$ generated by $T_q$, $S_q$ and $N_q$ will be denoted by $\PGL_{2,q}(\ZZ)$, it is isomorphic to $\PGL_2(\ZZ)$.

\begin{thm}\label{qdeformed_GL2}
There is a well-defined group homomorphism 
\begin{equation}
	\label{iso_pgl}
\begin{array}{c c c}
\PGL_2(\ZZ) & \longrightarrow & \PGL_{2,q}(\ZZ) \subset \PGL_2(\Lambda)\\
M & \longmapsto & M_q\\
\end{array},
\end{equation}
\noindent sending $T$, $S$, $N$ to $T_q$, $S_q$, $N_q$, and it is an isomorphism of groups.
\end{thm}

This theorem will be proved in Section~\ref{subsec:qPGL2}. We will also describe the traces of the $q$-matrices $M_q$, coming from matrices of determinant $-1$. We show that they are palindromic polynomials with coefficients of constant sign in Theorem~\ref{thm:palindrome}. In Appendix~\ref{appendixa}, we provide a combinatorial interpretation of traces of the $q$-matrices as generating functions of some fence posets.

Under the larger action of $\PGL_2(\ZZ)$ on $q$-deformed rational numbers, the left and right versions form a single orbit in $\PP^1(\ZZ(q))$. On irrational real numbers, the action of $\PGL_2(\ZZ)$ commutes with the deformation. The first part of the following theorem dealing with rational numbers will be proved in Section~\ref{subsec:qPGL2}, and the second part on real numbers will be proved in Section~\ref{subsec:convergence}. 

\begin{thm}
\label{sharp_and_flat}
Let $M\in \GL_2(\ZZ)$, with $\det(M) = -1$. Let $x\in \PP^1(\QQ)$. Then in $\Lambda$,
$$ 
M_q \cdot [x]^{\sharp} = [M\cdot x]^{\flat} \text{ and } M_q\cdot [x]^{\flat} = [M\cdot x]^{\sharp}.
$$
\noindent Furthermore, for $x\in \RR\setminus \QQ$, 
$$
M_q \cdot [x]_q = [M\cdot x]_q.
$$
\end{thm}

\subsection{Second extension of symmetries}

The next step is to extend again the group acting on $\PP^1(\ZZ(q))$, from $\PGL_2(\ZZ)$ to $\PGL_2(\ZZ)\times \ZZ_2$. To do this, we consider the operator 
$$
\begin{array}{c c c c}
\tau : & \PP^1(\ZZ(q)) & \longrightarrow & \PP^1(\ZZ(q))\\
		& f & \longrightarrow & f(q^{-1})\\
\end{array}.
$$

This involution will twist the action of every matrix $M\in \PGL_2(\ZZ)$. In the following, we will denote by $\circ $ the composition of operators of $\PP^1(\ZZ(q))$. We define 

\begin{equation}
I_q = \begin{pmatrix}
					1 & q-1\\
					1-q & q\\
					\end{pmatrix}.
\end{equation}

The operator $I_q\circ \tau$ is a deformation of the identity $\Id$ of $\PGL_2(\ZZ)$, and is conjugated to the $q$-rational transition map of Thomas in \cite{thomas_2024}, which makes transition between left and right $q$-rationals.

For any $M\in \PGL_2(\ZZ)$, we define the twisted $q$-deformation of $M$ to be the operator
$$
M_q \circ I_q\circ \tau : f(q) \longmapsto M_qI_q \cdot f(q^{-1}).
$$

\begin{thm}
\label{duplication}
The following map is a well-defined group isomorphism.
\begin{equation}
\begin{array}{c c c}
\PGL_2(\ZZ) \times \ZZ_2 & \longrightarrow & \PGL_{2,q}(\ZZ) \times \{\Id,I_q\cdot \tau \}\\
(M,1) & \longmapsto & (M_q,\Id) \\
(M,-1) & \longmapsto & (M_q,I_q\circ \tau)\\
\end{array}.
\end{equation}
\noindent It gives rise to a $q$-action 
\begin{equation}
\begin{array}{c c c}
\big(\PGL_{2}(\ZZ) \times \ZZ_2\big) \times \PP^1(\ZZ(q)) & \longrightarrow & \PP^1(\ZZ(q))\\
\big((M,\epsilon),f\big) & \longmapsto & \begin{cases}
										M_q \cdot f \hfill \text{ if } \epsilon = 1\\
										M_qI_q \cdot f(q^{-1})  \hfill \text{ if } \epsilon = -1\\									\end{cases}
\end{array}.
\end{equation}
\end{thm}

The twisted $q$-action of matrices of determinant $-1$ commutes with left and right deformations whereas in determinant $1$ they are exchanged.

\begin{thm}
\label{sharp_and_flat_bis}
Let $M \in \PGL_2(\ZZ)$. Let $x\in \PP^1(\QQ)$. If $\det(M) = -1$, then 
$$
M_qI_q \circ \tau \cdot [x]^{\sharp} = [M\cdot x]^{\sharp} \text{ and } M_qI_q\circ \tau \cdot [x]^{\flat} = [M\cdot x]^{\flat}.
$$
\noindent A contrario, if $\det(M) = 1$, then 
$$
M_qI_q\circ \tau \cdot [x]^{\sharp} = [M\cdot x]^{\flat} \text{ and } M_qI_q\circ \tau \cdot [x]^{\flat} = [M\cdot x]^{\sharp}.
$$
\end{thm}

Theorems \ref{sharp_and_flat_bis} and \ref{duplication} will be proved in Section~\ref{sec:second_extension}. 

\subsection{Application to irrational algebraic numbers}

As an illustration of what can be done with the $q$-action of $\PGL_{2}(\ZZ)$, we study deformations of algebraic equations. Even if we know how to deform irrational numbers \cite{MGOr}, there is still many interesting questions unsolved for $q$-deformed algebraic numbers.
 
The case of algebraic numbers of degree $2$ has already been fully understood by Leclere and Morier-Genoud \cite{Leclere_modular}. In \cite{Ovsienko_Ustinov}, Ovsienko and Ustinov partly covered the case of equations of degree $3$. Here we focus on degree $4$ and degree $6$, considering equations of the form

\begin{equation}
x^4 - bx^2 + 1 = 0 \text{ and } x^6 - 3x^5 - bx^4 + (2b + 5)x^3 - bx^2 - 3x + 1 = 0.
\end{equation}

Exploiting symmetries of these equations, we are able to quantize the Vieta's relations. It is the topic of Section~\ref{sec:algebraic_numbers}.

\section{Deformed action of the projective linear group} \label{sec:qPGL}

In this section, we give the details of the $q$-action of $\PGL_2(\ZZ)$ on $\PP^1(\ZZ(q))$ that we have introduced in Section 1.2 above.

\subsection{From $\PSL_2(\ZZ)$ to $\PGL_2(\ZZ)$} \label{subsec:qPGL2}

Recall the usual presentation by generators and relations of $\PGL_2(\ZZ)$.

\begin{lemme}
The group $\PGL_2(\ZZ)$ is isomorphic to 
$$
\langle T,S,N ~|~ S^2 = (TS)^3 = 1 \text{ , } N^2 = (NT)^2 = (NS)^2 = 1 \rangle, 
$$
\noindent with 
$$
T = \begin{pmatrix}
		1 & 1 \\
		0 & 1\\
		\end{pmatrix} \text{ , } S = \begin{pmatrix}
										0 & -1\\
										1 & 0\\
										\end{pmatrix} \text{ , } N = \begin{pmatrix}
	-1 & 0 \\
	0 & 1\\
	\end{pmatrix}.
$$
\end{lemme}

The first two generators and relations are those defining $\PSL_2(\ZZ)$, which have already been deformed in \cite{MGO} and \cite{Leclere_modular}. Thus we just have to define a $q$-version of $N$ and check the three last relations to have a $q$-deformation of $\PGL_2(\ZZ)$.\\
\medskip

We will work over the localization of $\ZZ[q,q^{-1}]$ by $t = 1-q+q^2$.
$$
\Lambda := \ZZ[q,q^{-1}]_t \text{ , } \overline{\Lambda} = \ZZ(q).
$$
In the group $\PGL_2(\Lambda)$, the matrices are defined up to sign, powers of $q$ and powers of $t$, which are the invertible elements of $\Lambda$.

\begin{defi}
Put $N_q = \begin{pmatrix}
-1 & 1-q^{-1}\\
q-1 & 1\\
\end{pmatrix} \in \PGL_2(\Lambda)$, and denote by $\PGL_{2,q}(\ZZ)$ the subgroup of $\PGL_2(\Lambda)$ generated by $T_q$, $S_q$ and $N_q$.
\end{defi}

\begin{thm} 
The group $\PGL_{2,q}(\ZZ)$ is isomorphic to $\PGL_2(\ZZ)$. 
\end{thm}

\begin{proof}
Straightforward computations show that matrices $T_q$, $S_q$ and $N_q$ satisfy
$$N_q^2 = (N_q T_q)^2 = (N_qS_q)^2 = t\Id.$$
This enables to define a group homomorphism $[\cdot]_q : \PGL_2(\ZZ) \rightarrow \PGL_2(\Lambda)$, mapping $T$ to $T_q$, $S$ to $S_q$ and $N$ to $N_q$. By definition, the image of $[\cdot]_q$ is $\PGL_{2,q}(\ZZ)$. \\
\noindent Now, if we denote by $\ev_1$ the map $\PGL_{2,q}(\ZZ) \longrightarrow \PGL_2(\ZZ)$ given by the evaluation $q=1$, we have $\ev_1 \circ [\cdot]_q = \id$. Indeed, for each generator $T$, $S$ and $N$ this is true, and this extends to the whole group. Therefore $[\cdot]_q$ is injective.
\end{proof}

\begin{ex}
	We can deduce a $q$-deformation of the operation $x\mapsto 1/x$ given by the matrix $J = \begin{pmatrix}
					0 & 1\\
					1 & 0\\
					\end{pmatrix}$ of determinant $-1$.
	$$ J = NS \overset{q}{\longrightarrow }
	J_q = N_qS_q = \begin{pmatrix}
		q-1 & 1 \\
		q & 1-q\\
	\end{pmatrix}.
	$$
\noindent Note that we can choose $T$ and $J$ to be generators of $\PGL_2(\ZZ)$, as we have the following relations
$$
T^{-1}JTJT^{-1} = S \text{ and } JTJT^{-1}JT = N.
$$
\noindent In the $q$-setting, these relations become
$$
T_q^{-1}J_qT_qJ_qT_q^{-1} = tS_q \text{ and } J_qT_qJ_qT_q^{-1}J_qT_q = tqN_q.
$$
\end{ex} 

This $q$-version of $\PGL_2(\ZZ)$ is the unique one compatible with $\PSL_{2,q}(\ZZ)$.

\begin{prop}
Let $\rho : \PGL_2(\ZZ) \longrightarrow \PGL_2(\Lambda)$ be a group homomorphism, and suppose that $\rho(T) = T_q$ and $\rho(S) = S_q$. Then $\rho(N) = N_q$.
\end{prop}

\begin{proof}
Denote $\rho(N) = \begin{pmatrix}
	a & b\\
	c & d\\
\end{pmatrix}$. The relation $N^2 = \Id$ imposes that $a = -d$.\\
\noindent Let us compute $\rho(NT)$ and $\rho(NS)$.
$$
\rho(NT) = \begin{pmatrix}
	a & b\\
	c & -a\\
\end{pmatrix} \begin{pmatrix}
	q & 1\\
	0 & 1\\
\end{pmatrix} = \begin{pmatrix}
	qa & a+b\\
	qc & c-a\\
\end{pmatrix}
$$
$$ \rho(NS) = \begin{pmatrix}
	a & b\\
	c & -a\\
\end{pmatrix} \begin{pmatrix}
	0 & -1\\
	q & 0\\
\end{pmatrix} = \begin{pmatrix}
	qb & -a\\
	-qa & -c\\
\end{pmatrix}.
$$
\noindent Then relations $(NT)^2 = \Id$ and $(NS)^2 = \Id$ in $\PGL_2(\ZZ)$ implies that $aq = -c+a $ and $qb = c$, so 
$$
\rho(N) = \begin{pmatrix}
	a & a(q^{-1}-1)\\
	a(1-q) & -a\\
\end{pmatrix}.
$$
\noindent Finally, $\det(\rho(N)) = -a^2(q^{-1}-1+q)$ must be a unit of $\Lambda$, so $a$ is a unit of $\Lambda$, and we can normalize $\rho(N)$ to get $N_q$. 
\end{proof}

Let us now consider the $q$-action
$$
\begin{array}{c c c}
\PGL_{2,q}(\ZZ) \times \PP^1(\ZZ(q)) & \longrightarrow & \PP^1(\ZZ(q))\\
(M_q,f) & \longmapsto & M_q\cdot f\\ 
\end{array}.
$$

First, recall that the action of $\PSL_{2,q}(\ZZ)$ commutes with both left and right deformations.

\begin{prop}[\cite{Leclere_modular,BBL}] \label{prop:equivariance_det1} Let $x\in \PP^1(\QQ)$ and let $M\in \PSL_2(\ZZ)$. Then 
$$
M_q \cdot [x]^{\flat} = [M\cdot x]^{\flat} \text{ and } M_q \cdot [x]^{\sharp} = [M\cdot x]^{\sharp}.
$$
\end{prop} 

We have a similar result for matrices of determinant $-1$. This is the first part of Theorem \ref{sharp_and_flat}.

\begin{prop}\label{prop:equivariance_detm1}
Let $M\in \PGL_2(\ZZ)$, with $\det(M) = -1$. Let $x\in \PP^1(\QQ)$. Then in $\Lambda$,
$$ 
M_q \cdot [x]^{\sharp} = [M\cdot x]^{\flat} \text{ and } M_q\cdot [x]^{\flat} = [M\cdot x]^{\sharp}.
$$
\end{prop}

\begin{proof}
We can write $M = M^+ N$ with $\det(M^+) = 1$, so that $M_q = M^+_q N_q$. Let us consider first the case $x = \infty$. Then
$$
M_q\cdot [\infty]^{\sharp} = M^+_qN_q \cdot \frac{1}{0} = M^+_q \cdot \frac{1}{1-q} = [M^+ \cdot \infty]^{\flat}.
$$
\noindent Now for $x\in \QQ$, there is a matrix $A \in \PSL_2(\ZZ)$ such that $A\cdot \infty = x$, and then
$$
M_q \cdot [x]^{\sharp} = M_q A_q \cdot [\infty]^{\sharp} = [M A \cdot \infty]^{\flat} = [M\cdot x]^{\flat}.
$$
\noindent The equality $M_q \cdot [x]^{\flat} = [M\cdot x]^{\sharp}$ is symmetric. 
\end{proof}

\begin{ex}
Let us take the matrix $N = \begin{pmatrix}
								-1 & 0\\
								0 & 1\\
								\end{pmatrix}$, and $x = n\in \NN^*$. Then
\begin{align*}
N_q \cdot [n]^{\sharp} = \frac{-[n]^{\sharp} + 1-q^{-1}}{(q-1)[n]^{\sharp} + 1} &= q^{-n}(-[n]^{\sharp} + 1-q^{-1})\\
&= -q^{-1} - q^{-2} - \cdots - q^{-n} + q^{-n} - q^{-(n+1)}\\
&= -q^{-1} - q^{-2} - \cdots -q^{-(n-1)} - q^{-(n+1)}\\
\end{align*}

\noindent On the other hand, 
$$
[N \cdot n]^{\flat} = [-n]^{\flat} = -q^{-1} - q^{-2} - \cdots - q^{-(n-1)} - q^{-(n+1)}.
$$
\noindent So $N_q\cdot [n]^{\sharp} = [N\cdot n]^{\flat}$. In the other way around, 

\begin{align*}
[N]_q \cdot [n]^{\flat} = \frac{-[n]^{\flat} + 1-q^{-1}}{(q-1)[n]^{\flat} + 1} &= -q^{-1}\frac{1+ q^2 + q^3 + \cdots + q^{n-1} + q^{n+1}}{q^{n+1}-q^n+q^{n-1}}\\
&= -q^{-1}\frac{1 - q^n + q^2[n]^{\sharp}}{q^{n-1}t}\\
&= -q^{-1}\frac{[n]^{\sharp}}{q^{n-1}}\\
&= -q^{-1}[n]^{\sharp}_{q^{-1}}\\
&= [-n]^{\sharp}.
\end{align*}

\noindent So $N_q \cdot [n]^{\flat} = [-n]^{\sharp}$.
\end{ex}

\begin{ex}
Let us consider now $x = \frac{3}{2}$ and $M = \begin{pmatrix}
	3 & 1 \\
	1 & 0\\
\end{pmatrix}$. The matrix $M$ decomposes as $M = T^3J$, so 
$$
M_q = T_q^3 J_q = \begin{pmatrix}
	q^4 + q^2 + q & 1 \\
	q & 1-q\\
\end{pmatrix}. 
$$
\noindent On one side, $\left[M\cdot \frac{3}{2}\right]^{\sharp} = \left[\frac{11}{3}\right]^{\sharp} = \frac{q^5 + 2q^4 + 2q^3 + 3q^2 + 2q + 1}{q^2 + q + 1}$. \\
\noindent On the other side, one can compute $\left[\frac{3}{2}\right]^{\flat} = \frac{1+q^2+q^3}{1+q^2}$. Then applying the $q$-matrix $M_q$ and cancelling a factor $t$ above and below the fraction, 
$$
M_q \cdot \left[\frac{3}{2}\right]^{\flat} = M_q \cdot \frac{1+q^2+q^3}{1+q^2} = \frac{q^5 + 2q^4 + 2q^3 + 3q^2 + 2q + 1}{q^2+q+1}.
$$

\end{ex}

\subsection{Deformed positive continued fractions} \label{subsec:continued_fractions}

The initial definition of $q$-rational numbers was based on continued fractions \cite{MGO}. Each rational number $u/v$ can be written as
$$
\frac{u}{v} = a_1 + \dfrac{1}{a_2 + \dfrac{1}{\ddots + \dfrac{1}{a_n}}} , \text{ with } a_1 \in \ZZ \text{ and } a_i \in \NN^*, i\geq 2,
$$

\noindent which is also denoted by $u/v = [a_1,a_2,\cdots,a_n]$. The sequence $(a_1,a_2,\cdots,a_n)$ is unique when the parity of $n$ is fixed. This continued fraction expansion is translated in the matrix setting by 
\begin{equation}
\label{continued_fraction}
\begin{pmatrix}
u & u'\\
v & v'\\
\end{pmatrix}  = \begin{pmatrix}
					a_1 & 1 \\
					1 & 0 \\
					\end{pmatrix}\begin{pmatrix}
						a_2 & 1 \\
						1 & 0 \\
						\end{pmatrix} \cdots \begin{pmatrix}
							a_n & 1 \\
							1 & 0 \\
							\end{pmatrix},
\end{equation}
\noindent where $u'/v' = [a_1,a_2,\cdots,a_{n-1}]$. To work with matrices of determinant $1$, it was necessary to choose $n$ even and to consider pairs of elementary matrices, using the relation 
$$
T^{a_1}J T^{a_2} J = T^{a_1}L^{a_2}, \text{ where } L = \begin{pmatrix}	
	1 & 0\\
	1 & 1\\
\end{pmatrix}.
$$  
With the deformation of $J$, we can state directly a $q$-deformation of equation \eqref{continued_fraction}, using the $q$-elementary matrices 
$$
T^a J = \begin{pmatrix}
	a & 1\\
	1 & 0\\
\end{pmatrix} \overset{q}{\longrightarrow} T_q^a J_q = \begin{pmatrix}
																q[a]^{\flat} & 1 \\
																q & 1-q\\
																\end{pmatrix}.
$$
Note that $\det(T_q^a J_q) = -q^{a-1}t$. 

\begin{prop}[Proposition 4.3 of \cite{MGO}]
	\label{prop_de_mgo}
Let $u/v$ be a rational number (with $u/v \notin \ZZ_{\leq 0}$) and let $[a_1,a_2,\cdots,a_{2m}]$ be its even positive continued fraction expansion. Let us set $u_{n-1}/v_{n-1} = [a_1,a_2,\cdots,a_{2m-1}]$ and 
$$
\left[\frac{u}{v}\right]^{\sharp} = \frac{U^{\sharp}}{V^{\sharp}} \text{ and } \left[\frac{u_{2m-1}}{v_{2m-1}}\right]^{\sharp} = \frac{U_{2m-1}^{\sharp}}{V_{2m-1}^{\sharp}}.
$$
\noindent Then it holds
$$
T_q^{a_1}J_q T_q^{a_2}J_q \cdots T_q^{a_{2m}}J_q = q^{\min(0,a_1)}t^{m}\begin{pmatrix}
												qU^{\sharp} & U_{2m-1}^{\sharp}\\
												qV^{\sharp} & V_{2m-1}^{\sharp}\\
												\end{pmatrix}.
$$
\end{prop}

\begin{rem}
For nonpositive integers, the formula above does not hold, and instead is 
$$
T_q^{-n-1}J_qT_qJ_q = q^{-n-1}t\begin{pmatrix}
	-q^2[n]^{\sharp} & -[n+1]^{\sharp}\\
	q^{n+2} & q^{n+1}\\
\end{pmatrix}.
$$
\end{rem}

We can also state a left version of Proposition \ref{prop_de_mgo} by choosing the lenght $n$ of the continued fraction expansion to be odd.

\begin{prop}
\label{prop:qmatrix_odd}
Let $u/v$ be a rational number (with $u/v\notin \ZZ$) and let $[a_1,a_2,\cdots,a_{2m+1}]$ its odd positive continued fraction expansion, with $m\geq 1$. Let us set $u_{2m}/v_{2m} = [a_1,\cdots,a_{2m}]$ and 
$$
\left[\frac{u}{v}\right]^{\flat} = \frac{U^{\flat}}{V^{\flat}} \text{ and } \left[\frac{u_{2m}}{v_{2m}}\right]^{\flat} = \frac{U_{2m}^{\flat}}{V_{2m}^{\flat}}.
$$
\noindent Then it holds
$$T_q^{a_1} J_q T_q^{a_2} J_q \cdots T_q^{a_{2m+1}}J_q = q^{\min(0,a_1)}t^{m} \begin{pmatrix}
																		qU^{\flat} & U_{2m}^{\flat}\\
																		qV^{\flat} & V_{2m}^{\flat}\\
																		\end{pmatrix}.
$$
\end{prop}

\begin{proof} We consider $u_{2m-1}/v_{2m-1} = [a_1,\cdots,a_{2m-1}]$, and its right quantization $[u_{2m-1}/v_{2m-1}]^{\sharp} = U_{2m-1}^{\sharp}/V_{2m-1}^{\sharp}$. According to the even case, 
$$
T_q^{a_1} J_q T_q^{a_2} J_q \cdots T_q^{a_{2m}}J_q = q^{\min(0,a_1)}t^{m}\begin{pmatrix}
	qU_{2m}^{\sharp} & U_{2m-1}^{\sharp}\\
	qV_{2m}^{\sharp} & V_{2m-1}^{\sharp}\\
	\end{pmatrix}.
$$

\noindent Therefore 
$$
T_q^{a_1} J_q T_q^{a_2} J_q \cdots T_q^{a_{2m}}J_qT_q^{a_{2m+1}}J_q = q^{\min(0,a_1)}t^{m}\begin{pmatrix}
								q(q[a_{2m+1}]^{\flat}U_{2m}^{\sharp} + U_{2m-1}^{\sharp}) & qU_{2m}^{\sharp} + (1-q)U_{2m-1}^{\sharp}\\
								q(q[a_{2m+1}]^{\flat}V_{2m}^{\sharp} + V_{2m-1}^{\sharp}) & qV_{2m}^{\sharp} + (1-q)V_{2m-1}^{\sharp}\\
								\end{pmatrix}.
$$

\noindent An induction on $m$ shows that 
$$
U^{\flat} = q[a_{2m+1}]^{\flat}U_{2m}^{\sharp} + U_{2m-1}^{\sharp} \text{ and } V^{\flat} = q[a_{2m+1}]^{\flat}V_{2m}^{\sharp} + V_{2m-1}^{\sharp}.
$$

\noindent Finally, as indicated in Appendix A of \cite{BBL}, $qU_{2m}^{\sharp} + (1-q)U_{2m-1}^{\sharp} = U_{2m}^{\flat}$ and same for denominators.
\end{proof}

\begin{rem}
For an integer $n$, the odd continued fraction expansion is $[n]$, and we have 
$$
T_q^nJ_q = \begin{pmatrix}
	q\widetilde{[n]^{\flat}} & 1\\
	q & 1-q\\
\end{pmatrix},
$$
\noindent where $\widetilde{[n]^{\flat}}$ is the symmetric of $[n]^{\flat}$.
\end{rem}

\begin{ex}
Let us detail the case of $\frac{u}{v} = \frac{7}{5} = [1,2,1,1] = [1,2,2]$. With the even continued fraction, we get 
$$
T_qJ_qT_q^2J_qT_qJ_qT_qJ_q = t^2\begin{pmatrix}
q^5 + 2q^4 + 2q^3 + q^2 + q   &  q^3 + q^2 + q + 1\\
q^4 + 2q^3 + q^2 + q    & q^2 + q + 1\\
\end{pmatrix},
$$
\noindent so that $\left[\frac{7}{5}\right]^{\sharp} = \frac{1+q+2q^2+2q^3+q^4}{1+q+2q^2+q3}$. For the odd continued fraction, 
\begin{align*}
T_qJ_qT_q^2J_qT_q^2J_q &= \begin{pmatrix}
	q^8 + 2q^6 + 2q^4 + q^3 + q  &    q^5 + 2q^2 - q + 1\\
 q^7 + q^5 + q^4 + q^3 + q   &  q^4 - q^3 + 2q^2 - q + 1\\
\end{pmatrix} \\
&= t\begin{pmatrix}
q^6 + q^5 + 2q^4 + q^3 + q^2 + q &     q^3 + q^2 + 1\\
 q^5 + q^4 + q^3 + q^2 + q  &          q^2 + 1\\
\end{pmatrix}.\\
\end{align*}

\noindent thus $\left[\frac{7}{5}\right]^{\flat} = \frac{1+q+q^2+2q^3+q^4+q^5}{1+q+q^2+q^3+q^4}$.
\end{ex}

\subsection{Palindromicity and positivity of $q$-traces in determinant $-1$} \label{subsec:palindromic}

One of the main properties of $q$-deformed matrices of $\PSL_2(\ZZ)$ is that their traces are palindromic polynomials \cite{Leclere_modular}. This property still holds for the $q$-deformed matrices of determinant $-1$. 

\begin{thm}
\label{thm:palindrome}
Let $M\in \PGL_2(\ZZ)$, with $\det(M) = -1$. Then $\Tr(M_q)$ is a palindromic polynomial with positive integer coefficients, up to a sign and powers of $q$ and $t=q^2-q+1$. \\
\end{thm}

Our proof uses negative continued fractions, also called Hirzebruch-Jung continued fractions (see \cite{Hirzebruch1}). For a rational number $x\in \QQ$, there are uniquely determined integers $c_1 \in \ZZ$, $c_2,\cdots,c_k \geq 2$, such that 
$$
x = c_1 - \dfrac{1}{c_2 - \dfrac{1}{\ddots - \dfrac{1}{c_k}}}.
$$
This negative continued fraction expansion of $x$ is denoted by $x = [\![c_1,c_2,\cdots,c_k]\!]$.

\begin{nota}
    Let $c_1,c_2,\cdots,c_k$ be integers. We denote
    $$
    M(c_1,c_2,\cdots,c_k) := T^{c_1}ST^{c_2}S\cdots T^{c_k}S \in \PSL_2(\ZZ).
    $$
    \noindent We also denote by $M_q(c_1,c_2,\cdots,c_k)$ the $q$-deformed version of this matrix (we choose the representant in $\GL_{2}(\ZZ[q,q^{-1},t^{-1}])$ given by the $q$-generators as in \eqref{qgenerators}). 
\end{nota}

We divide the proof of Theorem \ref{thm:palindrome} in two parts, the palindromicity on one hand and the positivity on the other hand. We adapt the proofs of Leclere and Morier-Genoud \cite{Leclere_modular} that considered matrices of determinant $1$. In particular, we fix a minor error in their statement of positivity. \\
~\\
\noindent \underline{Proof of the palindromicity}. If $\det(M)=1$, the palindromicity of $\Tr(M_q)$ is a result of Leclere and Morier-Genoud \cite[Theorem 3]{Leclere_modular}. \\
\noindent If $\det(M)=-1$, there are integers $c_1,c_2,\cdots,c_k\in \ZZ$ such that $M = NT^{c_1}ST^{c_2}S\cdots T^{c_k}S = NM(c_1,c_2,\cdots,c_k)$. Let us compute $^tM_{q^{-1}}$.
\begin{align*}
^tM_{q^{-1}} &= \begin{pmatrix}
				[c_k]_{q^{-1}}^{\sharp} & 1 \\
				-q^{1-c_k} & 0 \\
				\end{pmatrix} \cdots 
				\begin{pmatrix}
				[c_2]_{q^{-1}}^{\sharp} & 1 \\
				-q^{1-c_2} & 0 \\
				\end{pmatrix} \begin{pmatrix}
				[c_1]_{q^{-1}}^{\sharp} & 1 \\
				-q^{1-c_1} & 0 \\
				\end{pmatrix} ~ ^tN_{q^{-1}} ~\\
	&= q^{-\Sigma (c_i - 1)}\begin{pmatrix}
				[c_k]_{q}^{\sharp} & q^{c_k-1} \\
				-1 & 0 \\
				\end{pmatrix} \cdots 
				\begin{pmatrix}
				[c_2]_{q}^{\sharp} & q^{c_2-1} \\
				-1 & 0 \\
				\end{pmatrix} \begin{pmatrix}
				[c_1]_{q}^{\sharp} & q^{c_1-1} \\
				-1 & 0 \\
				\end{pmatrix} ~^tN_{q^{-1}}\\
	&= q^{-\Sigma (c_i - 1)}NM_q(c_k,\cdots,c_2,c_1)N^tN_{q^{-1}},\\
\end{align*}
\noindent because the conjugation by $N$ turns the signs of the antidiagonale to the opposite. And $N^tN_{q^{-1}} = N_qN$, so taking the trace we get:
$$
\Tr(M_{q^{-1}}) = q^{-\Sigma (c_i - 1)}\Tr\big(M_q(c_k,\cdots,c_2,c_1)N_q\big).
$$
\noindent It is thus enough to show that $\Tr(M_q(c_k,\cdots,c_2,c_1)N_q) = \Tr(N_qM_q(c_1,c_2,\cdots,c_k))$. Let us set 
$$
M_q(c_1,c_2,\cdots,c_k) = \begin{pmatrix}
							A & B \\
							C & D\\
\end{pmatrix} \text{ and } 
M_q(c_k,\cdots,c_2,c_1) = \begin{pmatrix}
					\tilde{A} & \tilde{B}\\
					\tilde{C} & \tilde{D}
					\end{pmatrix}.
$$
\noindent With these notations, $\Tr(N_qM_q(c_1,c_2,\cdots,c_k)) = -A +(1-q^{-1})C + (q-1)B + D$, and same for $\Tr(N_qM_q(c_k,\cdots,c_2,c_1))$ with $\tilde{A},\tilde{B},\tilde{C},\tilde{D}$. In the proof of Lemma 3.8 of \cite{Leclere_modular}, the following relations are settled, using induction on the length of the continued fraction $[\![c_1,\cdots,c_k]\!]$ and on the integer $c_k$:
$$
\begin{cases}
A + D = \tilde{A} + \tilde{D}	~~~~~~(1)\\
C - qB = \tilde{C} - q\tilde{B}~~~~ (2)\\
A + B - C = \tilde{A} + \tilde{B} - \tilde{C} ~~ (3)
\end{cases}.
$$
\noindent Taking $(2) + (3)$, we get \\
\noindent $A - (q-1)B = \tilde{A} - (q-1)\tilde{B}$, and taking $(1) - q^{-1}(2) - (3)$, we get $D + (1-q^{-1})C = \tilde{D} + (1-q^{-1})\tilde{C}$, and this shows exactly that 
$$ 
\Tr(N_qM_q(c_1,c_2,\cdots,c_k)) = \Tr(N_qM_q(c_k,\cdots,c_2,c_1)).
$$
\noindent Therefore $\Tr(M_{q^{-1}}) = q^{- \Sigma (c_i-1)}\Tr(M_q)$, meaning that $\Tr(M_q)$ is a palindromic polynomial.\\
~\\
\noindent \underline{Proof of the positivity}. Because Theorem \ref{thm:palindrome} states a result up to powers of $q$ and sign, in the rest of the paragraph we work up to these. To emphasize this convention, we will write $A \equiv B$ when two Laurent polynomials $A$ and $B$ are equal up to a power of $q$ and a sign.\\
\noindent Besides, we write $A\geq 0$ to say that the Laurent polynomial $A$ has positive coefficients. \\
~\\
\noindent We begin with a corrected version of Lemma 3.10 in \cite{Leclere_modular}.

\begin{lemme}\label{lemma:positive_trace}
Let $x = \negfrac{c_1,c_2,\cdots,c_k}$ be the negative continued fraction of a rational number, with $x\geq 1$ so that the $c_i's$ are integers greater than $2$. Then $\Tr(M_q(c_1,c_2,\cdots,c_k)) \geq 0$. 
\end{lemme} 

\begin{proof}
We start an induction on $k$, aiming to prove that 
\begin{center}
$(H)$ \text{     } $\forall k \geq 1,~~ \forall (c_1,\cdots,c_k) \in \ZZ^k_{\geq 2} $, if $M_q(c_1,\cdots,c_k) = \begin{pmatrix}
		A & B \\
		C & D\\
	\end{pmatrix}$, we have
	$$A \geq 0 \text{  ,  } -B \geq 0 \text{  ,  } A+D \geq 0 \text{  and  } A+B \geq 0.$$
\end{center}

\noindent For $k = 1$,
$$
M_q(c_1) = \begin{pmatrix}
	[c_1]^{\sharp} & -q^{c_1-1}\\
	1 & 0\\
\end{pmatrix}.
$$
\noindent And $(H)$ is satisfied.\\
~\\
\noindent Let $k \geq 2$ and let us suppose that $(H)$ is true for any $k_0 \leq k$. Let $c_1,\cdots,c_k,c$ be $k+1$ integers greater or equal to $2$. We set 
$$
M_q(c_1,\cdots,c_k) = \begin{pmatrix}
	A & B \\
	C & D\\
\end{pmatrix},
$$
\noindent so that by hypothesis we have $A \geq 0$, $-B\geq 0$,  $A+D \geq 0 \text{ and } A+B \geq 0$.\\
\noindent Moreover, as the $c_i$'s are greater than $2$, the rational number $x = \negfrac{c_1,c_2,\cdots,c_k}$ is greater than $1$, and by total positivity of $q$-rational numbers, we can deduce that $A-C \geq 0$. \\
~\\
\noindent One can compute 
$$
M_q(c_1,\cdots,c_k,c) = \begin{pmatrix}
	[c]^{\sharp} A + B & -q^{c-1}A\\
	[c]^{\sharp} C + D & -q^{c-1}C\\
\end{pmatrix} = \begin{pmatrix}
	A' & B'\\
	C' & D'\\
\end{pmatrix}.
$$
\noindent Then $A' = q[c-1]^{\sharp}A + (A + B) \geq 0$, $-B' = q^{c-1}A \geq 0$, and 
\begin{align*}
A' + D' &= q[c-1]^{\sharp}A + A + B - q^{c-1}C\\
		&= q[c-2]^{\sharp}A + q^{c-1}(A-C) + (A+B)\\
		&\geq 0.
\end{align*}

\noindent Likewise, 
\begin{align*}
A' + B' &= [c]^{\sharp}A + B - q^{c-1}A\\
		&= q[c-2]^{\sharp}A + (A + B)\\
		& \geq 0.
\end{align*}

\noindent This concludes the induction and the proof of the lemma.
\end{proof}

In general, a matrix $M$ with $\det(M)=1$ can be written as $M = M(c_1,c_2,\cdots,c_k)$, and Proposition~7.4 of \cite{MGO_Farey} ensures that we can suppose $c_3,\cdots,c_{k-2} \geq 2$, and $1\leq c_1\leq c_2$, and $1\leq c_k \leq c_{k-1}$. The trace of $M_q(c_1,c_2,\cdots,c_k)$ is invariant by circular permutation of the $c_i$'s, and $M_q(1,1,1) = -q^3\Id$, so we can even suppose that $c_2,\cdots,c_{k} \geq 2$ and $1\leq c_{1}$. \\
~\\
\noindent $\bullet$ Case 1:  $c_{1} \geq 2$. By the previous Lemma \ref{lemma:positive_trace}, $\Tr(M_q)$ has positive coefficients. \\
~\\
\noindent $\bullet$ Case 2: $c_1 = 1$ and $k=1$ or $c_1=1$, $c_2=2$ and $k=2$. Then we can check that $\Tr(T_qS_q) = q$ and $\Tr(T_qS_qT_q^2S_q)= 0$ have positive coefficients. \\
~\\
\noindent $\bullet$ Case 3: $c_1 = 1$ and $k\geq 2$. Using the relation $T_qS_qT_q = -qS_qT_q^{-1}S_q$ and circular permutations, we have $\Tr(M_q) \equiv \Tr(M_q(c_2-1,c_3,\cdots,c_{k-1},c_k-1))$, so by induction on $k$ the result follows.\\
~\\
It remains to consider the case of a matrix $M$ with $\det(M)=-1$. Again, we begin with a particular case. 
\begin{lemme}\label{lemma:negative_trace}
Let $x  =\negfrac{c_1,c_2,\cdots,c_k} \in \QQ_{\geq 0}$, so $c_i\geq 2$ for $i\in \negfrac{2,k}$ and $c_1\geq 1$. \\
\noindent Then $-\Tr(N_qM_{q}(c_1,c_2,\cdots,c_k)) \geq 0$.
\end{lemme}

\begin{proof}
Denote $[x]^{\sharp} = \frac{A}{B}$ with $A$ and $B$ two coprime polynomials. Since $x\geq 0$, the denominator $B$ has constant coefficient $1$. Let the left $q$-deformation of $-x$ be $\frac{-A^{\flat}}{B^{\flat}}$, with $A^\flat$ and $B^\flat$ two coprime polynomials. Since $-x \leq 0$, the numerator $-A^\flat$ has constant coefficient $-1$. By Theorem \ref{sharp_and_flat} applied to $N_q$, we have $N_q\cdot \frac{A}{B} = \frac{-qA + (q-1)B}{(q^2-q)A + qB} = \frac{-A^\flat}{B^\flat}$, and matching valuations we get
$$
-A^\flat = -qA + (q-1)B \text{ and } B^\flat = (q^2-q)A + qB.
$$
\noindent Then by \cite[Proposition 4.3]{MGO}, 
$$
N_qM_q(c_1,c_2,\cdots,c_k) = q^{k-1}\begin{pmatrix}
	-A^\flat & q^{c_k-1}A_{k-1}^\flat\\
	B^\flat & -q^{c_k-1}B^\flat_{k-1}\\
\end{pmatrix},
$$
\noindent where $\frac{-A_{k-1}^\flat}{B_{k-1}^\flat}$ is the left $q$-deformation of $-x' = -\negfrac{c_1,c_2,\cdots,c_{k-1}}$.\\
\noindent By the positivity property of left $q$-rationals, $A^\flat$ and $B_{k-1}^\flat$ have positive coefficients, and then $\Tr(M_q) = -q^{k-1}A^\flat -q^{k+c_k-2}B^\flat_{k-1}$ has negative coefficients. 
\end{proof}

In general, a matrix $M$ with determinant $-1$ can be written as $M = NM(c_1,c_2,\cdots,c_k)$, with $c_3,\cdots,c_{k-2} \geq 2$ and $1\leq c_1\leq c_2$, $1\leq c_k\leq c_{k-1}$. To manage circular permutations, we will use $M_q(c)N_q = -q^cN_qM_q(-c)$, implying that
$$
\Tr(N_qM_q(c_1,c_2,\cdots,c_k)) \equiv \Tr(N_qM_q(c_2,\cdots,c_k,-c_1)).
$$
\noindent We start an induction on $k$. If $k=1$ or $k=2$, direct computations show the result. Suppose then $k\geq 3$.\\
\noindent $\bullet$ Case 1: $c_k \geq 2$. If $c_2 \geq 2$, Lemma \ref{lemma:negative_trace} above applies. Suppose $c_2 = c_1 = 1$. We can apply relation $T_qS_qT_qS_qT_q = q^2S_q$, and get $\Tr(M_q) \equiv \Tr(N_qS_qM_q(c_3-1,\cdots,c_k))$. If $k=3$, we can conclude easily, and if $k\geq 4$, acording to whether $c_k\geq c_3-1$ or not and using circular permutation, we get $\Tr(N_qM_q(c_4,\cdots,c_{k-1},c_k-c_3+1))$ or $\Tr(N_qM_q(c_3-c_k-1,c_4,\cdots,c_{k-1}))$ and we conclude by induction on $k$. \\
~\\
\noindent $\bullet$ Case 2: $c_{k-1} \geq 2$ and $c_k = 1$. By applying $S_qT_qS_q = q^2T^{-1}_qS_qT^{-1}_q$ and a circular permutation, we get $\Tr(M_q) = \Tr(N_qM_q(c_1+1,c_2,\cdots,c_{k-2},c_{k-1}-1))$ and we conclude by induction. \\
~\\
\noindent $\bullet$ Case 3: $c_{k-1} = c_k = 1$. \\
\noindent By $(T_qS_q)^3 = -q^3\Id$, we have $M_q(c_1,\cdots,c_k)= M_q(c_1,\cdots,c_{k-3})
T_q^{c_{k-2}-1}$. Thus, according to whether $c_1\geq c_{k-2}-1$ or not, we apply a circular permutation to get $\Tr(N_qM_q(c_1-c_{k-2}+1,c_2,\cdots,c_{k-3}))$ or $\Tr(N_qM_q(c_2,\cdots,c_{k-2}-1-c_1))$ and conclude by induction.\\
~\\
\noindent Positivity is proved, and so is Theorem \ref{thm:palindrome}.

\begin{ex}
Let us consider the matrix $M = \begin{pmatrix}
	-7 & 3 \\
	5 & -2
\end{pmatrix}$, which is decomposed on $T,S,N$ as 
$$
M = NT^2ST^2ST^3S.
$$ 

\noindent After deformation, we get 
$$
\Tr(M_q) = \Tr \begin{pmatrix}
	-q^5 - q^4 - q^3 - 2q^2 - q - 1 & q^5 + q^3 + q^2\\
	q^6 + q^5 + q^4 + q^3 + q^2 & -q^6 - q^4\\
\end{pmatrix}    = - (1 + q + 2q^2 + q^3 + 2q^4 + q^5 + q^6).
$$
\noindent This is a palindromic polynomial, but notice that it is not unimodular. 
\end{ex}

\section{Duplication of $\PGL_2(\ZZ)$} \label{sec:second_extension}

The aim is to define an action of $\PGL_2(\ZZ)\times \ZZ_2$ on $\PP^1(\ZZ(q))$, containing the previous $q$-deformed action of $\PGL_2(\ZZ)$. 

\subsection{Twisted matrices} \label{subsec:twist}

Let us consider the operator 
$$
\begin{array}{c c c c}
\tau : & \PP^1(\ZZ(q)) & \longrightarrow & \PP^1(\ZZ(q))\\
		&  f & \longrightarrow & f(q^{-1})\\
\end{array}.
$$

\begin{prop}
Up to powers of $t = q^2 - q +1$, there is only one matrix $I_q$ in $\PGL_2(\Lambda)$ such that the operator $I_q \circ \tau$ commutes with the quantized action of $\PGL_2(\ZZ)$ on $\PP^1(\ZZ(q))$.
\noindent This matrix is 
$$
I_q = \begin{pmatrix}
								1 & q-1 \\
								1-q & q\\
								\end{pmatrix},
$$
and it satisfies $(I_q\circ \tau)^2 = \Id$. 
\end{prop}

\begin{proof}
Let $I_q$ be given by $\begin{pmatrix}
			a & b \\
			c & d\\
\end{pmatrix}$, and suppose that $I_q \circ \tau$ commutes with every element $A_q$ for $ A \in \PGL_{2}(\ZZ)$. In particular,  
$$
I_q\circ \tau \circ T_q = T_q \circ I_q \circ \tau.
$$
This means that up to a power of $q$, 
$$
\begin{pmatrix}
aq^{-1} & a+b \\
cq^{-1} & c+d\\
\end{pmatrix} = \begin{pmatrix}
							aq+c & qb+d\\
							c & d\\
							\end{pmatrix}.
$$
Putting $c= 0$ leads to $q=q^{-1}$, so we can suppose that $c\neq 0$.\\
\noindent Then the bottom left coefficient imposes that the right hand side is $q$ times the left hand side, so $\begin{cases}
			a = aq+c \\
			qa + qb = qb +d\\
			qc+qd = d\\
			\end{cases}$. \\
\noindent Then $c = (1-q)a$, and $d = qa$,
$$I_q = \begin{pmatrix}
			a & b \\
			(1-q)a & qa\\
			\end{pmatrix}.
$$
\noindent Similarly with $S_q$, we get up to a power of $q$ that 
$$
\begin{pmatrix}
b & -qa \\
qa & (q^2-q)a\\
\end{pmatrix} = \begin{pmatrix}
							(1-q^{-1})a & -a\\
							a & b\\
							\end{pmatrix},
$$
\noindent hence $b = (q-1)a$. Finally, we can check that $I_q\circ \tau \circ N_q = N_q \circ I_q\circ \tau $. \\
~\\
\noindent To have $\det(I_q)$ invertible in $\Lambda$, the coefficient $a$ must be invertible so $a$ is a power of $tq$.
\end{proof}

\begin{defi}
We say that $I_q\circ \tau$ is the twisted deformation of the identity of $\PGL_2(\ZZ)$.\\
\noindent Then, for every matrix $M\in \PGL_2(\ZZ)$, we define the twisted deformation of $M$ as the operator
$$
M_qI_q \circ \tau \text{ acting on } \PP^1(\ZZ(q)).
$$ 
\end{defi}

\begin{rem}
The twisted deformation of identity already appeared in \cite{thomas_2024}, where Thomas defined the $q$-rational transition map $g_q : x \mapsto \frac{1+(x-1)q}{1+(q-1)x}$. The transition map and the twisted identity are related by
$$
\tau \circ g_q = I_q \circ \tau.
$$
\end{rem}

\begin{rem}
As $(I_q\circ \tau)^2 = \Id$ in $\PGL_{2}(\Lambda)$, for all elements $A, B$ in $\PGL_2(\ZZ)$, we have 
$$
(A_qI_q \circ\tau) \circ (B_qI_q\circ \tau) = A_qB_q.
$$
\end{rem}

\begin{ex}
The generators $T,S$ and $N$ have twisted $q$-deformations as follows (for convenience we do not write the precomposition by $\tau$):
$$
T_qI_q = \begin{pmatrix}
					1 & q^2 \\
					1-q & q \\
					\end{pmatrix}, ~ S_qI_q = \begin{pmatrix}
					q-1 & -q \\
					q & q^2-q \\
					\end{pmatrix} \text{ and } N_qI_q = \begin{pmatrix}
					-t & 0 \\
					0 & qt \\
					\end{pmatrix} = \begin{pmatrix}
					-1 & 0\\
					0 & q\\
					\end{pmatrix}.
$$
\noindent Recall the other usual generator $L = TST$ ($L$ together with $T$ generate $\PSL_2(\ZZ)$). In $\PGL_{2}(\Lambda)$, it becomes 
$$
L_q = \begin{pmatrix}
1 & 0\\
1 & q^{-1}\\
\end{pmatrix} .
$$
\noindent We can check that 
$$
(T_qI_q\circ \tau)\circ (S_qI_q\circ \tau)\circ (T_qI_q\circ\tau) = \begin{pmatrix}
					q & q^2-q \\
					1 & q^2 \\
					\end{pmatrix} \circ \tau = L_qI_q\circ \tau.
$$ 

\noindent For the matrix $J = \begin{pmatrix}
								0 & 1 \\
								1 & 0\\
								\end{pmatrix}$, the twisted deformation is trivial:
$$
J_qI_q = \begin{pmatrix}
				0 & 1\\
				1 & 0\\
				\end{pmatrix}.
$$
\end{ex}

\subsection{Action on rational numbers} \label{subsec:twist_on_numbers}

\begin{defi}
	We define the $q$-deformed duplication of $\PGL_2(\ZZ)$ to be the action 
	$$
	\begin{array}{c c c}
	\big(\PGL_{2,q}(\ZZ)\times \ZZ_2\big)\times \PP^1(\ZZ(q)) & \longrightarrow & \PP^1(\ZZ(q))\\
	\big((M_q,\epsilon), f\big) & \longmapsto & \begin{cases}
								M_q \cdot f  \text{ if } \epsilon = 1\\
								M_qI_q \cdot f(q^{-1}) \text{ if } \epsilon = -1\\
								\end{cases}.\\
	\end{array}
	$$
	\end{defi}

\begin{prop}
Let $M \in \PGL_2(\ZZ)$. Let $x\in \PP^1(\QQ)$. Then \\
\noindent if $\det(M) = 1$,
$$
M_qI_q \cdot [x]^{\sharp}_{q^{-1}} = [M\cdot x]^{\flat}_q \text{ and } M_qI_q \cdot [x]^{\flat}_{q^{-1}} = [M\cdot x]^{\sharp}_q,
$$

\noindent and if $\det(M) = -1$, 
$$
M_qI_q \cdot [x]^{\sharp}_{q^{-1}} = [M\cdot x]^{\sharp}_q \text{ and } M_qI_q \cdot [x]^{\flat}_{q^{-1}} = [M\cdot x]^{\flat}_q.
$$
\end{prop}

\begin{proof}
\noindent These equalities are direct consequences of Propositions \ref{prop:equivariance_det1} and \ref{prop:equivariance_detm1}, and of the behaviour of $I_q\circ \tau$: 
$$
I_q \cdot [x]^{\sharp}_{q^{-1}} = [x]^{\flat}_q \text{ and } I_q \cdot [x]^{\flat}_{q^{-1}} = [x]^{\sharp}_q.
$$
\end{proof}

\begin{ex}
Let us take $M = T$ and $x = n\in \NN^*$. We have 
\begin{align*}
T_qI_q \cdot [n]^{\sharp}_{q^{-1}} = \begin{pmatrix}
	1 & q^2  \\
	1-q & q\\
\end{pmatrix} \cdot [n]^{\sharp}_{q^{-1}} &= \frac{[n]^{\sharp}_{q^{-1}}+q^2}{(1-q)[n]^{\sharp}_{q^{-1}}+q} \\
&= q^{n-1}[n]^{\sharp}_{q^{-1}} + q^{n+1}\\
&= [n]^{\sharp}_q + q^{n+1}\\
&= [n+1]^{\flat}_q.
\end{align*}
\end{ex}

\begin{coro}
Let $x \in \PP^1(\QQ)$. Then 
$$
[-x]^{\sharp}_q = -q^{-1}[x]^{\sharp}_{q^{-1}} = \frac{-[x]_q^{\flat} + 1 - q^{-1} }{(q-1)[x]_q^{\flat} + 1}.
$$

$$
\left[\frac{1}{x}\right]^{\sharp}_q = \frac{1}{[x]_{q^{-1}}^{\sharp}} = \frac{(q-1)[x]_q^{\flat} +1}{q[x]_q^{\flat} + 1-q}.
$$
\end{coro}

\section{Focus on $q$-real numbers}\label{sec:qreals}

\subsection{Convergences} \label{subsec:convergence}

Recall the definition of $q$-deformed real numbers, in \cite{MGOr}. This takes place in the ring $\ZZ[[q]][q^{-1}]$ of formal Laurent series in $q$, 

$$
\ZZ[[q]][q^{-1}] = \left\lbrace \sum_{k\geq \nu} x_{k}q^k ~|~ x_k \in \ZZ, ~ \nu \in \ZZ \right\rbrace,
$$

\noindent endowed with its usual metric relying on the valuation $\nu$ of formal Laurent series,
$$
d(f,g) = \frac{1}{2^{\nu(f-g)}}.
$$

Thus saying that a sequence $\left(x_n = \sum_{k\geq \nu_n} x_{n,k}q^k \right)_n$ in $\ZZ[[q]][q^{-1}]$ converges toward $y = \sum_{k\geq \nu} y_k q^k$ means exactly that the valuations $\nu_n$ stabilize to $\nu$ and the coefficients $x_{n,k}$ stabilize to $y_k$ as $n$ grows,
$$\forall k\in \ZZ, \exists N\in \NN, \forall n\geq N, \forall j\leq k, x_{n,j} = y_j.$$

\begin{prop}[\cite{MGOr}]
\label{right_convergence}
Let $x \in \RR \setminus \QQ$ be an irrational number. Then for any sequence of rational numbers $(x_n)_n$ converging to $x$, the Taylor expansions of $[x_n]^{\sharp}$ converge toward a Laurent series that does not depend on the choice of the sequence $(x_n)_n$.
\end{prop}

\begin{defi}
With the previous notations, the resulting series of the convergence of the Taylor expansions of $([x_n]^{\sharp})_n$ is defined to be the deformation of $x$ and is denoted by $[x]_q$. 
\end{defi}

For irrational numbers, the left and right deformations are the same. 

\begin{prop}
Let $x \in \RR \setminus \QQ$ be an irrational number. Let $(x_n)_n$ be a sequence of rational numbers converging to $x$. Then, the sequence $([x_n]^{\sharp} - [x_n]^{\flat})_n$ converges to the zero series. 
\end{prop}

\begin{proof}
We can suppose that $(x_n)_n$ is actually the sequence of convergents of $x$: $x_n = [a_0,a_1,a_2,\cdots, a_n]$, with $a_0\in \ZZ$ and $a_i \in \ZZ_{\geq 1}$ for all $i \geq 1$. \\

\noindent Let us write, for all $n \in \NN$, $[x_n]^{\sharp} = \pm q^{d}\frac{\mathcal{R}_n(q)}{\mathcal{S}_n(q)}$, with $\mathcal{R}_n$ and $\mathcal{S}_n$ two coprime monic polynomials in $q$ with constant coefficient $1$. The integer $d$ depends on $x$ as 
$$
d = \begin{cases}
		0 \text{ if } x \geq 1\\
		\lfloor 1/x \rfloor \text{ if } 0 < x <1\\
		\lfloor x \rfloor \text{ if } x < 0\\
		\end{cases}.
$$

\noindent The left version of $x_n$ is then
$$ [x_n]^{\flat} = \pm q^{d}~ \frac{q\mathcal{R}_n(q)+(1-q)\mathcal{R}_{n-1}(q)}{q\mathcal{S}_n(q) + (1-q)\mathcal{S}_{n-1}(q)},$$

\noindent so 
$$[x_n]^{\sharp} - [x_n]^{\flat} = \pm q^{d}\frac{(1-q)(\mathcal{R}_n(q)\mathcal{S}_{n-1}(q) - \mathcal{R}_{n-1}(q)\mathcal{S}_n(q))}{q\mathcal{S}_n(q)^2 + (1-q)\mathcal{S}_n(q)\mathcal{S}_{n-1}(q)} = \frac{(1-q)q^{d + a_1+\cdots + a_n -1}}{q\mathcal{S}_n(q)(\mathcal{S}_n(q)-\mathcal{S}_{n-1}(q)) + \mathcal{S}_{n-1}(q)}.$$

\noindent And the series $1/(q\mathcal{S}_n(q)(\mathcal{S}_n(q)-\mathcal{S}_{n-1}(q))+\mathcal{S}_{n-1}(q))$ starts with the coefficient $1$ thanks to $\mathcal{S}_{n-1}$, so the series $[x_n]^{\sharp} - [x_n]^{\flat}$ starts with $q^{d + a_1+\cdots a_n -1}$, and thus converges to the zero series when $n$ grows.
\end{proof}

\begin{coro}
\label{left_convergence}
Let $x\in \RR \setminus \QQ$ be an irrational number. Let $(x_n)_n$ be a sequence of rational numbers converging to $x$. Then the sequence of left $q$-deformed numbers $([x_n]^{\flat})_n$ converges to $[x]_q$.
\end{coro}

The following corollary is the second part of Theorem \ref{sharp_and_flat}. 

\begin{coro}
Let $x\in \RR \setminus \QQ$ be an irrational number. Let $M \in \PGL_2(\ZZ)$, with $\det(M) = -1$. Then 
$$M_q \cdot [x]_q = [M\cdot x]_q.$$
\end{coro}

\begin{proof}
Let $(x_n)_{n\geq 0}$ be any sequence of rational numbers converging to $x$. For all $n \geq 0$, by the first part of Theorem \ref{sharp_and_flat},
$$
M_q \cdot [x_n]^{\sharp} = [M \cdot x_n]^{\flat}.
$$
Then by Proposition \ref{right_convergence} and continuity, $M_q \cdot [x_n]^{\sharp}$ converges toward $M_q \cdot [x]_q$ in the sense of Laurent series, and by Corollary \ref{left_convergence} $[M\cdot x_n]^{\flat}$ converges toward $[M\cdot x]_q$. Therefore
$$
M_q\cdot [x]_q = [M\cdot x]_q.
$$
\end{proof}

\begin{ex} For every irrational number $x$, we have the formulas
$$ 
[-x]_q = [N]_q \cdot [x]_q \text{ and } \left[\frac{1}{x}\right]_q = [N S]_q \cdot [x]_q.
$$ 

More concretely,
$$[-x]_q = \frac{-[x]_q + 1-q^{-1}}{(q-1)[x]_q + 1} \text{ and } \left[\frac{1}{x}\right]_q = \frac{(q-1)[x]_q + 1}{q[x]_q + 1 -q}.$$

\noindent These formulas could not be obtained without a quantization of the $\PGL_2(\ZZ)$-action. In \cite{Leclere_modular}, it has already been noticed that for $x\in \QQ$,
$$
\left[\frac{1}{x}\right]^{\sharp} = \frac{1}{[x]_{q^{-1}}^{\sharp}} = \overline{J_q}\cdot [x]^{\sharp},
$$
\noindent but this formula could not be generalized to the case of irrational numbers, as there is no clear way of defining the operator $\tau $ on formal Laurent series.
\end{ex}

\begin{ex}
The radius of convergence of the series $\left[\frac{1}{\pi}\right]_q$ is the same as the one of $[\pi]_q$, and one can compute the terms of the quantized series of $1/\pi$ with the formula
$$
\left[\frac{1}{\pi}\right]_q = \frac{(q-1)[\pi]_q + 1}{q[\pi]_q + 1-q}.
$$
The first 30 terms of the series $[1/\pi]_q$ are 
$$
\begin{array}{c}
q^{3} -q^{5} -q^{6} + q^{7} + 2q^{8} -4q^{10} -q^{11} + 5q^{12} + 5q^{13} -6q^{14} -11q^{15} + 3q^{16} + 20q^{17} + 6q^{18} -28q^{19} \\
-26q^{20} + 31q^{21} + 58q^{22} -17q^{23} -103q^{24} -28q^{25} + 146q^{26} + 131q^{27}  -165q^{28} -299q^{29} + 94q^{30}...
\end{array}
$$

\end{ex}

Finally, the terminology of left and right quantizations comes from the following.

\begin{prop}[\cite{MGOr,BBL}] 
Let $x\in \QQ$ be a rational number.
\begin{itemize}
	\item[(i)] Let $(x_n)_n$ be a sequence of rational numbers converging to $x$ by the right. Then the sequences of right and left $q$-deformed numbers $([x_n]^{\sharp})_n$ and $([x_n]^{\flat})_n$ converge to $[x]^{\sharp}$.
	\item[(ii)] Let $(x_n)_n$ be a sequence of rational numbers converging to $x$ by the left. Then the sequences of right and left $q$-deformed numbers $([x_n]^{\sharp})_n$ and $([x_n]^{\flat})_n$ converge to $[x]^{\flat}$.
\end{itemize}
\end{prop}

\subsection{Injectivity of deformations} \label{subsec:injectivity}

We are proving Proposition \ref{injective} for the right deformation. The proof is the same in the left case. So let us consider the right deformation map
$$
\mathcal{Q}^{\sharp} : \RR \longrightarrow \ZZ[[q]][q^{-1}].
$$

On the rational numbers, it is injective, because one can evaluate at $q = 1$ to recover the rational number from its deformation.

Let $x$ and $x'$ be two irrational numbers such that $[x]_q = [x']_q$. The continued fractions of $x$ and $x'$ are infinite:
$$
x = [a_1,a_2,a_3,\cdots ] \text{ , } x' = [a_1',a_2',a_3',\cdots ].
$$

\noindent By Theorem 2 of \cite{MGOr}, we know that the integer part of a real number can be read directly on its $q$~-~deformation, so we get that $a_1 = a_1'$. \\
\noindent Then put 
$$
y = \frac{1}{x-a_1} = [a_2,a_3,\cdots ]  \text{ and  } y' = \frac{1}{x'-a_1} = [a_2',a_3', \cdots].
$$ 
\noindent We have $y = JR^{-a_1}\cdot x$ and $y' = JR^{-a_1}\cdot x'$,  so 
$$[y]_q = [JR^{-a_1}]_q \cdot [x]_q = [JR^{-a_1}]_q \cdot [x']_q = [y']_q.$$

\noindent Again $a_2 = a_2'$, and by induction for all $i$, $a_i = a_i'$, so $x = x'$. Proposition \ref{injective} is proved.

\section{Application to $q$-deformed algebraic numbers}\label{sec:algebraic_numbers}

As an illustration of the extended $q$-actions on $\PP^1(\ZZ(q))$, we deform some algebraic equations, in a compatible way with the $q$-deformations of their roots (being thus algebraic real numbers).

\subsection{Degree 4}

We consider the following particular case of algebraic equation of degree $4$, with $b$ an integer:
\begin{equation}
\label{eq4}
x^4 - bx^2 + 1 = 0.
\end{equation}

\noindent We want to study the $q$-deformations of the solutions and the relations they share. Assuming there are four distinct real solution $x_1$, $x_2$, $x_3$ and $x_4$, we denote by $\sigma_1$, $\sigma_2$, $\sigma_3$ and $\sigma_4$ the elementary symmetric polynomials:
\begin{itemize}
\item $\sigma_1 = x_1 + x_2 + x_3 + x_4 = 0$;
\item $\sigma_2 = x_1x_2 + x_1x_3 + x_1x_4 + x_2x_3 + x_2x_4 + x_3x_4 = -b$;
\item $\sigma_3 = x_1x_2x_3 + x_1x_2x_4 + x_1x_3x_4 + x_2x_3x_4 = 0$ ;
\item $\sigma_4 = x_1x_2x_3x_4 = 1$.
\end{itemize}

\begin{nota}
For $i \in \{1,2,3,4\}$, the $q$-deformed number $[x_i]_q $ will be denoted $X_i$, and the elementary polynomials in the $X_i$'s will be denoted by $\Sigma_1$, $\Sigma_2$, $\Sigma_3$, $\Sigma_4$.
\end{nota}

\begin{lemme}
The equation \eqref{eq4} has four distinct real solutions if and only if $b > 2$. \\
\noindent Moreover, the polynomial $x^4 -bx^2 +1$ is irreducible over $\QQ$ except when $b = n^2 \pm 2$, for $n \in \NN$.
\end{lemme}

\begin{prop}
Let us suppose that $b > 2$ in equation \eqref{eq4}. Then the $q$-deformed numbers $X_1$, $X_2$, $X_3$ and $X_4$ are solutions of the following quantized equation
$$X^4 - \Sigma_1X^3 + \left(\frac{q+q^{-1}-3}{q-1}\Sigma_1 + 2q^{-1}\right)X^2 + q^{-1}\Sigma_1X + q^{-2} = 0,$$
\noindent where $\Sigma_1 = X_1 + X_2 + X_3 + X_4$ is a Laurent series in $q$.\\
\noindent In other words, the Vieta's relations become
\begin{itemize}
\item $\Sigma_4 = q^{-2}$, 
\item $\Sigma_3 = -q^{-1}\Sigma_1$, 
\item $(q-1)\Sigma_2 = (q + q^{-1} - 3)\Sigma_1 + 2(1 - q^{-1}).$
\end{itemize}
\end{prop}

\begin{proof}
\noindent As the equation \eqref{eq4} is invariant under the two transformations $x\mapsto -x$ and $x\mapsto 1/x$, the matrices $N$ and $J$ act on the roots.

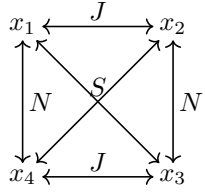
\begin{figure}[H]
\label{galois4}
\centering
\begin{tikzpicture}[auto,node distance=2cm,on grid,semithick,inner sep=2pt,bend angle=45]

\node (X1) {$x_1$} ;
\node (X2) [right =of X1] {$x_2$};
\node (X4) [below =of X1] {$x_4$};
\node (X3) [below =of X2] {$x_3$};

\path
(X1) edge[<->] node {$J$} (X2)
	edge [<->] node {$N$} (X4)
	edge [<->] (X3)
(X4) edge [<->] node {$J$} (X3)
	edge [<->,above] node {$S$} (X2)
(X2) edge [<->] node {$N$} (X3);
\end{tikzpicture}
\caption{Galois group action on the roots by fractional-linear transformations}
\end{figure}

\noindent After quantization, $[J]_q$ and $[N]_q$ also act on the $X_i$'s, and then 
$$\begin{pmatrix}
	-1 & 1-q^{-1}\\
	q-1 & 1\\
	\end{pmatrix} X_4 = X_1 \text{ so } (q-1)X_1X_4 = -X_1 -X_4 + 1-q^{-1},$$
\noindent and same for $X_2$ and $X_3$: $(q-1)X_2X_3 = -X_2 -X_3 + 1-q^{-1}$. 
\noindent Letting $[J]_q$ and $[S]_q$ act, we get four more equations:
$$qX_1X_2 = (q-1)(X_1 + X_2) + 1 \text{ , }  qX_3X_4 = (q-1)(X_3+X_4) + 1 \text{ , } qX_1X_3 = -1 \text{ and } qX_2X_4 = -1.$$

\noindent Using these equations, we get the quantized Vieta's relations. For instance, 
\begin{align*}
(q-1)\Sigma_2 &= (q-1)(X_1X_2 + X_3X_4 + X_1X_4 + X_2X_3 + X_2X_4 + X_1X_3)\\
&= q^{-1}(q-1)^2(X_1+X_2+X_3+X_4) + 2q^{-1}(q-1) \\
&+ (-X_1-X_2-X_3-X_4 +2(1-q^{-1})) -2q^{-1}(q-1)\\
&= (q-2+q^{-1})\Sigma_1 - \Sigma_1 + 2(1-q^{-1})\\
&= (q-3 + q^{-1})\Sigma_1 + 2(1-q^{-1}),
\end{align*}
\noindent and similarly for the other Vieta's relations.
\end{proof}

\begin{rem} 
The Galois group of the polynomial in equation \eqref{eq4} is isomorphic to the group $\{\Id, J,S,N\}$ which is the Klein four group, acting on the roots $x_1,x_2,x_3,x_4$ exactly as shown in Figure \ref{galois4}. \\
\noindent Indeed, if $\sigma \in \Gal_{\QQ}(x^4-bx^2+1)$, then it is determined by the image of one root $x_1$ of $x^4-bx^2+1$ because all the roots are algebraic expressions of $x_1$.
\end{rem}

\begin{ex}
Let us consider more specifically the case $b = n^2 - 2$, for $n \geq 3$. Then the polynomial splits in two quadratic polynomials
$$x^4 -(n^2 - 2)x^2 + 1 = (x^2 -nx + 1)(x^2 +nx + 1).$$
\noindent The solutions $X_1,X_2,X_3,X_4$ are then quadratic irrational numbers. Following \cite{Leclere_modular}, these quantized numbers are solutions in radicals of quantized quadratic equations.
For instance, let us detail the way $x^2 -nx + 1$ is quantized. The two roots $x_1, x_2$ of this polynomial satisfy $x_i^2 = nx_i -1$, so they are fixed points of the matrix $M = \begin{pmatrix}
		n & -1\\
		1 & 0\\
	\end{pmatrix} \in \PSL_2(\ZZ)$. In terms of the two generators $R$ and $S$ of the modular group (see \eqref{generators}), the matrix $M$ is $T^nS$. Then the $q$-deformed numbers $X_1$, $X_2$ are solutions of a quadratic equation given by the equality $M_q \cdot X_i = X_i$, where 
$$
M_q = T_q^n S_q = \begin{pmatrix}
					[n]^{\sharp} & -q^{n-1}\\
					1 & 0\\
					\end{pmatrix}.
$$	
\noindent That is, the equation $x^2 -nx + 1 = 0$ is quantized by 	 
$$
X^2 - [n]^{\sharp}X + q^{n-1} = 0.
$$
\noindent Likewise, the equation $x^2 +nx +1$ is quantized by 
$$
X^2 - [-n]^{\sharp}X + q^{-n-1} = 0.
$$

\noindent Therefore, the Vieta's formulas become:
$$
 \Sigma_4 = q^{-2} \text{ , } \Sigma_3 = q^{n-1}[-n]^{\sharp} + q^{-n-1}[n]^{\sharp},
$$
$$ \Sigma_2 = [n]^{\sharp}[-n]^{\sharp} + q^{-n-1}+q^{n-1} \text{ and } \Sigma_1 = [n]^{\sharp} + [-n]^{\sharp}.
$$

\noindent The case $b = n^2 + 2$ is similar, except that here we will get a matrix with determinant $-1$. Indeed, the decomposition is 
$$
x^4 - (n^2 + 2)x^2 + 1 = (x^2 -nx -1)(x^2 +nx-1),
$$
and the roots $x_1$, $x_2$ of $x^2 -nx -1$ are fixed points of the matrix $\begin{pmatrix}
						n & 1 \\
						1 & 0\\
						\end{pmatrix}$, which is $NT^{-n}S$. Its deformation is 
$$
\begin{pmatrix}
[n]^{\flat} & q^{-1} \\
1 & q^{-1}-1\\
\end{pmatrix}.
$$

\noindent Therefore the quantized equations are 
$$
X^2 + (q^{-1}-1 - [n]^{\flat})X - q^{-1} = 0 \text{ and } X^2 + (q^{-1}-1 - [-n]^{\flat})X - q^{-1} = 0, 
$$
\noindent and the Vieta's relations for the equation $x^4 - (n^2 +2)x^2 +1 = 0$ become
$$
\Sigma_4 = q^{-2} \text{ , } \Sigma_3 =  2(q^{-2} - q^{-1}) -q^{-1}([n]^{\flat} + [-n]^{\flat}) ,
$$
$$ \Sigma_2 = [n]^{\flat}[-n]^{\flat} -(q^{-1}-1)([n]^{\flat}+ [-n]^{\flat})+ 1 - 4q^{-1} + q^{-2} \text{ and } \Sigma_1 = 2(1-q^{-1}) + [n]^{\flat} + [-n]^{\flat}.
$$
\end{ex}

\subsection{Degree 6}

We consider the following equation, depending in one integer parameter $b$:
\begin{equation}
\label{eq6}
x^6 - 3x^5 - bx^4 + (2b + 5)x^3 - bx^2 - 3x + 1 = 0.
\end{equation}

We keep the same notations as in the case of degree $4$, the $q$-deformed solutions are still denoted $X_i$ ($1\leq i \leq 6 $) and the symmetric polynomials $\Sigma_i$'s. 

\begin{lemme} The equation \eqref{eq6} has six distinct real solutions if and only if $b \geq 1$. In this case, it is irreducible over $\QQ$ if and only if $b \neq 2$ and $b \notin \{k^2+k+1~|~k\geq 0\}$. 
\end{lemme}

\begin{proof}
Set $f(x) = x^6 - 3x^5 - bx^4 + (2b + 5)x^3 - bx^2 - 3x + 1$. \\
\noindent This polynomial is invariant under the action of $J$ and of $\Gamma := \begin{pmatrix}
				1 & -1 \\
				1 & 0\\
\end{pmatrix}$. So the set of roots is also stable by these transformations, and the roots are $x_1, \Gamma(x_1),\Gamma^2(x_1), 1/x_1, 1/\Gamma(x_1), 1/\Gamma^2(x_1)$. Therefore either all the roots are real, either all the roots are nonreal complex numbers. Then it is straightforward to check that the roots are real if and only if $b\geq 1$. \\
~\\
\noindent Now let us suppose that $b\geq 1$. If the polynomial $f$ is not irreducible over $\QQ$, it has either two factors of degree $3$ or three factors of degree $2$. \\
\noindent $\bullet$ Let us suppose that $f$ splits into two factors of degree $3$. Then these two factors must be 
$$
(x-x_1)(x-\Gamma(x_1))(x-\Gamma^2(x_1)) \text{ and } \left(x-\frac{1}{x_1}\right)\left(x-\frac{1}{\Gamma(x_1)}\right)\left(x-\frac{1}{\Gamma^2(x_1)}\right).
$$
\noindent So there is an integer $a$ such that $x_1\Gamma(x_1) + x_1\Gamma^2(x_1) + \Gamma(x_1)\Gamma^2(x_1) = a$, that is
$$
x_1^3 - (3+a)x_1^2 + ax_1 + 1 = 0,
$$
\noindent so the polynomial $f$ splits in two factors that are exaclty the one considered in \cite{Ovsienko_Ustinov},
$$
f(x) = (x^3 - (3+a)x^2 + ax + 1)(x^3 +ax^2 - (3+a)x + 1)~~~~ (\star).
$$
\noindent Identifying coefficient by coefficient, we get that $a$ must satisfy $a^2 + 3a + 3 -b = 0$, therefore $4b-3$ must be a square integer, that is $b = k^2 + k+1$, for some $k\in \NN$.\\
\noindent Conversely, if $b = k^2 +k+1$ for a $k\in \NN$, then the factorization $(\star)$ works, with $a = k-1$.\\
~\\
\noindent $\bullet$ Let us suppose that $f$ splits into three factors of degree $2$. Then one can check that the only possible configuration is when $b = 2$ and 
$$
f(x) = (x^2 -3x + 1)(x^2 -x-1)(x^2 +x-1).
$$
\noindent Thus $f$ is reducible over $\QQ$ if and only if $b = 2$ or $b = k^2+k+1$ for some $k\in \NN$.
\end{proof}

\begin{prop}
\label{Vieta6}
Let us suppose that $b \geq 1$. After the $q$-deformation of the six roots of \eqref{eq6}, the Vieta's relations are quantized as 
\begin{itemize}
\item $\Sigma_6 = 1$,
\item $\Sigma_5 = -\Sigma_1 + 6$,
\item $\Sigma_4 = -5\Sigma_1 + \Sigma_2 + 15$,
\item $\Sigma_3 = -5\Sigma_1 + 2\Sigma_2 + 10$, 
\item $(q-1)\Sigma_2 = (q^2 +q -4 +q^{-1})\Sigma_1 + 3(-q^2 +q +2 -q^{-1})$.
\end{itemize}

\end{prop}

\begin{proof}
We use the symmetries of \eqref{eq6}, which is invariant under the action of $J$ and $\Gamma$.

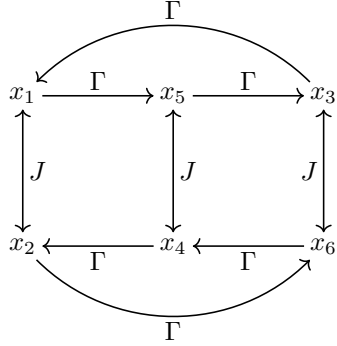
\begin{figure}[H]
\label{galois6}
\centering
\begin{tikzpicture}[auto,node distance=2cm,on grid,semithick,inner sep=2pt,bend angle=45]

\node (X1) {$x_1$} ;
\node (X5) [right =of X1] {$x_5$};
\node (X3) [right =of X5] {$x_3$};
\node (X2) [below =of X1] {$x_2$};
\node (X4) [right =of X2] {$x_4$};
\node (X6) [right =of X4] {$x_6$};

\path
(X1) edge[->] node {$\Gamma$} (X5)
	edge [<->] node {$J$} (X2)
(X5) edge [->] node {$\Gamma$} (X3)
	 edge [<->] node {$J$} (X4)
(X4) 	edge [->] node {$\Gamma$}  (X2)
(X6) edge [->] node {$\Gamma$} (X4)
	 edge [<->] node {$J$} (X3)
(X3) edge[->,bend right,above] node {$\Gamma$} (X1)
(X2) edge[->,bend right,below] node {$\Gamma$} (X6) ;
\end{tikzpicture}
\caption{Galois group action on the roots by fractional-linear transformations}
\end{figure}
 
\noindent After quantization, these transformations give
$$\begin{array}{c c c c}
X_1 X_2 = q^{-1}(q-1)(X_1+X_2) + q^{-1}, & X_1X_5 = X_1 - 1, & X_2X_6 = X_2 -1, & X_1X_4 = q(X_1+X_4)+1-q,\\
X_4 X_5 = q^{-1}(q-1)(X_4+X_5) + q^{-1}, & X_3X_5 = X_5 - 1, & X_4X_6 = X_6 - 1, & X_2X_3 = q(X_2+X_3)+1-q,\\
X_3 X_6 = q^{-1}(q-1)(X_3+X_6) + q^{-1}, & X_1X_3 = X_3 - 1, & X_2X_4 = X_4 -1, & X_5X_6 = q(X_5+X_6)+1-q.\\
\end{array}$$

\noindent The $q$-deformed Vieta's relations follow.
\end{proof}

\begin{ex}
Let us first take $b = 2$, so that the equation splits into $(x^2 -3x + 1)(x^2 -x-1)(x^2 +x-1) = 0$. The quantization of this equation comes from the quantization of quadratic equations, and it is 
$$
\big(X^2 - [3]^{\sharp}X + q^2\big)\big(qX^2 - (q^2+q-1)X -1\big)\big(q^2X^2 - (q^2-q-1)X -q\big) = 0.
$$
\noindent Therefore in this case, $\Sigma_1 = q+2q^2+3q^3+2q^4+q^5$ is finite, and we can check all the Vieta's relations of Proposition \ref{Vieta6}.
\end{ex}

\begin{ex}
Now let us take $b = 3$ so that the equation splits into
$$
(x^3 -3x^2 + 1)(x^3 -3x + 1) = 0.
$$
\noindent According to \cite{Ovsienko_Ustinov}, these two factors are $q$-deformed as 
$$
(X^3 - B_1(q)X^2 + (B_1(q) - 3)X + 1)(X^3 - B_2(q)X^2 + (B_2(q) - 3)X + 1) = 0, 
$$
\noindent with $B_1(q) = X_1(q) + X_3(q) + X_5(q)$ (resp. $B_2(q) = X_2(q) + X_4(q) + X_6(q)$) the sum of odd (resp. even) roots. \\
\noindent We can check again that the relations of Proposition \ref{Vieta6} are satisfied. Note that unlike the quadratic case, here it seems that the series $B_1(q)$ and $B_2(q)$ are actually infinite.
\end{ex}

\begin{rem}
As all roots of the polynomial defining \eqref{eq6} are algebraic expressions of one (arbitrarily chosen) root $x_1$, the Galois group is isomorphic to the subgroup of $\mathfrak{S}_6$ generated by $(1~2)(3~6)(4~5)$ and $(1~3~5)(2~6~4)$, and its action on the roots is realized by the fractional-linear transformations generated by $J$ and $\Gamma$, as depicted in Figure \ref{galois6}. 
\end{rem}

\appendix
\section{Fence poset combinatorics} \label{appendixa}

One of the key properties of traces of $q$-deformed matrices of $\PSL_{2,q}(\ZZ)$ is that they can be realized as generating functions, at least in the positive case. This property was the main tool used by Oğuz and Ravichandran in \cite{oguz_rank_2023} to prove a conjecture of Morier-Genoud and Ovsienko about unimodality of $q$-traces (see also \cite{McConville}). For matrices of determinant $-1$, we find the same kind of underlying combinatorics, with slight modifications.

\begin{defi}
Let $(a_1,\cdots,a_{2m+1})$ be a sequence of positive integers, and denote $N = \sum_{i=1}^{2m+1} a_i$. We define the left fence poset associated to this sequence, denoted here by $F^{\flat}(a_1,\cdots,a_{2m+1})$, to be the poset with $N+1$ elements $v_0, v_1, \hdots, v_N$ and cover relations 
\[
v_0 \triangleleft v_1 \triangleleft v_2 \triangleleft \dots \triangleleft v_{a_1} \triangleright v_{a_1+1} \triangleright \dots \triangleright v_{a_1+a_2} \triangleleft v_{a_1+a_2+1} \triangleleft \cdots \triangleright v_{N-a_{2m+1}} \triangleleft \dots \triangleleft v_{N-1} \triangleleft v_N.
\]

\noindent The circular left fence poset $\overline{F}^\flat(a_1,\hdots,a_{2m+1})$ is then obtained from the left fence poset $F^\flat(a_1,\hdots,a_{2m+1})$ by adding one more covering relation: $v_{0} \triangleright v_{N-a_{2m+1}+1}$.
\end{defi}

\begin{rem}
	We extend the definition to include sequences of even lengths, in the following way: considering a sequence $(a_1,\hdots,a_{2m})$ of positive integers, if $a_{2m} \geq 1$ then by definition $F^{\flat}(a_1,\cdots,a_{2m}) := F^{\flat}(a_1,\cdots,a_{2m}-1,1)$, and if $a_{2m} = 1$ then by definition $F^{\flat}(a_1,\cdots,a_{2m}) := F^{\flat}(a_1,\cdots,a_{2m-1}+1)$. 
\end{rem}

\begin{defi}
With the same notations as above, an admissible ideal of the left circular fence poset $\overline{F}^\flat(a_1,\hdots,a_{2m+1})$ is a subset $I$ such that 
\begin{itemize}
	\item[$(i)$] for all $x\in I$, if $y\leq x$ then $y \in I$ ;
	\item[$(ii)$] if $v_0 \in I$, then $v_N \in I$ iff $v_{N-1} \in I$ ;
	\item[$(iii)$] if $v_0 \notin I$, then $v_{N-a_{2m+1}} \in I $ iff $v_{N-a_{2m+1}-1} \in I $.
\end{itemize}
\end{defi}

\begin{ex}
	\label{ex_poset}
With $(a_1,a_2,a_3) = (2,2,4)$, the left circular poset is
\begin{center} 
\begin{tikzpicture}[scale=0.6]
		\draw[fill=black] (3,-3) circle (2pt);
	 	\node at (3,-3.6) {$_0$};
		\draw[line width=1pt] (0,1)-- (1,2);
		\draw[line width=1pt] (1,2)-- (3,0);
		\draw[line width=1pt] (3,0)-- (7,4);
		\draw[line width=1pt] (0,1)--(3,-3);
		\draw[line width=1pt] (3,-3)to[bend right] (4,1);
		\draw[fill=black] (0,1) circle (2pt);
		 \node at (0,0.4) {$_1$};
		\draw[fill=black] (1,2) circle (2pt) ;
		 \node at (1,2.6) {$_2$};
		\draw[fill=black] (2,1) circle (2pt);
		 \node at (2.3,1.6) {$_3$};
		\draw[fill=black] (3,0) circle (2pt) ;
		 \node at (3,-0.6) {$_4$};
		\draw[fill=black] (4,1) circle (2pt) ;
		 \node at (4,1.6) {$_5$};
		\draw[fill=black] (5,2) circle (2pt) ;
		 \node at(5,2.6) {$_6$};
		\draw[fill=black] (6,3) circle (2pt) ;
		 \node at(6,2.4) {$_7$};
		\draw[fill=black] (7,4) circle (2pt);
		 \node at(7,3.4) {$_8$};
\end{tikzpicture}
\end{center}
\noindent There is only one admissible ideal of size $1$ in this case, which is $\{0\}$. \\
\noindent There are three admissible ideals of size $2$, which are $\{0,1\} ; \{0,4\} ; \{3,4\}$. 
\end{ex}

\begin{prop}
\label{generative_function}
Let $(a_1,\cdots,a_{2m+1})$ be a sequence of positive integers. \\
\noindent Let $M_q := T_q^{a_1} J_q T_q^{a_2} J_q \cdots T_q^{a_{2m+1}}J_q$ be the associated $q$-deformed matrix. Then 
\[
\Tr(M_q) = t^{m} \sum_{I \subset \overline{F}^\flat(a_1,\hdots,a_{2m+1}) \atop \text{admissible ideal}} q^{\abs{I}}. 
\]
\end{prop}

\begin{proof}
	Let us denote $\frac{u}{v} = [a_1,a_2,\hdots,a_{2m+1}]$ and $\frac{u_{2m}}{v_{2m}} = [a_1,a_2,\hdots,a_{2m}]$ the previous convergent. Consider the left $q$-deformations of these two rationals 
\[
\left[\frac{u}{v}\right]^{\flat} = \frac{U^{\flat}}{V^{\flat}} \text{ and } \left[\frac{u_{2m}}{v_{2m}}\right]^{\flat} = \frac{U_{2m}^{\flat}}{V_{2m}^{\flat}}.
\]
By Proposition~\ref{prop:qmatrix_odd}, we have $\Tr(M_q) = t^{m}(qU^\flat + V_{2m}^\flat)$. The polynomials $U^\flat$ and $V_{2m}^\flat$ both can be interpreted as generating functions of some ideals in some fence posets. More precisely, by \cite[Corollary A.2]{BBL}, and using the conventions of \cite{Aval_Labbe}, 
\begin{itemize}
	\item[1.] $U^\flat $ is the generating function of order ideals $I$ in $F_1 := F^\flat(a_1,\hdots,a_{2m+1})$ such that $v_0\in I$ and $v_N \in I$ iff $v_{N-1} \in I$
	\item[2.] $V_{2m}^\flat $ is the generating function of order ideals in $F_2 := F^\flat(a_1,\hdots,a_{2m})$ such that $v_0\notin I$ and $v_{N-a_{2m+1}} \in I$ iff $v_{N-a_{2m+1}-1} \in I$.
\end{itemize}
\noindent This means that $qU^\flat$ is the generating function of admissible ideals in $\overline{F}^\flat(a_1,\hdots,a_{2m+1})$ containing $v_0$, and $V^\flat $ is the generating function of admissible ideals in $\overline{F}^\flat(a_1,\hdots,a_{2m+1})$ not containing $v_0$. Since the generating function of the disjoint union of two sets is the sum of the generating functions of the two sets, this concludes the proof. 
\end{proof}

\begin{ex}
Consider again the sequence $(2,2,4)$. The associated matrix is 
$$
M_q = T_q^2J_qT_q^2J_qT_q^4J_q = t\begin{pmatrix}
q^9 + q^8 + 3q^7 + 3q^6 + 4q^5 + 4q^4 + 3q^3 + 2q^2 + q  & q^4 + q^3 + q^2 + q + 1 \\
q^7 + q^6 + q^5 + 2q^4 + 2q^3 + q^2 + q  & q^2 + 1\\
\end{pmatrix},
$$
\noindent so $\Tr(M_q) = t(q^9 + q^8 + 3q^7 + 3q^6 + 4q^5 + 4q^4 + 3q^3 + 3q^2 + q + 1)$.\\
\noindent This corresponds to the generating function of admissible ideals of the poset described in Example~\ref{ex_poset} above.
\end{ex}

\section{Proof of Proposition \ref{two_deformations}} \label{appendixb}

We explain here why we consider only two $q$-deformations of the rational projective line, the left and the right versions, corresponding to the orbits in $\PP^1(\QQ(q))$ of $\frac{1}{0}$ and $\frac{1}{1-q}$ respectively. \\
~\\
Let $f(q) \in \PP^1(\QQ(q))$ be such that $f(1) = \frac{1}{0} \in \PP^1(\QQ)$. Let $O_f$ be the orbit of $f(q)$ under the action of the projective modular space $\PSL_{2,q}(\ZZ)$. We are interested in the following proposition:
\begin{center}
For all $x\in \QQ$, there is exactly one element $g_x \in O_f$ such that $g_x(1) = x$. 
\end{center}
~\\
\noindent This amounts to requiring that, via the isomorphism \eqref{iso_pgl} between $\PSL_2(\ZZ)$ and $\PSL_{2,q}(\ZZ)$, 
\begin{equation}
\label{stab_condition}
\Stab_{\PSL_{2,q}(\ZZ)}(f(q)) \simeq \Stab_{\PSL_2(\ZZ)}\left(\frac{1}{0}\right) = \langle T \rangle.
\end{equation}

\noindent We need then to show that this equality \eqref{stab_condition} holds if and only if $f \in \{\frac{1}{0}, \frac{1}{1-q}\}$. \\
Let us take $f$ such that \eqref{stab_condition} holds. This means that $T_q\cdot f = f$, so $qf + 1 = f$. Let us write $f$ as a quotient of two coprime polynomials $f = \frac{h}{g}$. \\
\noindent We get $\frac{qh+g}{g} = \frac{h}{g}$, so $g(qh+g) = hg$. Then either $g = 0$ meaning that $f = \frac{1}{0}$, either $g\neq 0$ and $(1-q)h = g$, meaning that $f = \frac{1}{1-q}$. \\
~\\
\noindent Conversely, the two solutions $f = \frac{1}{0}$ and $f = \frac{1}{1-q}$ satisfy \eqref{stab_condition}.

\section*{Acknowledgements} 

The author would like to thank their advisor Sophie Morier-Genoud, and Valentin Ovsienko, who guided them all the way long. They are also thankful to the anonymous referee who provided valuable feedback on the first version of the paper. This work is part of the PhD thesis of the author \cite{thesis_jouteur}. 


\bibliographystyle{plain}
\bibliography{full_biblio}

\end{document}